\documentclass[11pt, oneside]{amsart}

\usepackage{etex}
\usepackage[T1]{fontenc}
\usepackage[utf8]{inputenc}
\usepackage[english]{babel}
\usepackage{geometry}
\usepackage[pdftex]{graphicx}

\usepackage{amsmath}
\usepackage{amsfonts}
\usepackage{amssymb}
\usepackage{amsthm}
\usepackage{booktabs}
\usepackage{paralist}
\usepackage{subfig}
\usepackage{array}
\usepackage{xy}
\usepackage{multicol}
\usepackage{fancyhdr}
\usepackage{wrapfig}
\usepackage{bm}
\usepackage{listings}
\usepackage{enumerate}
\usepackage{multirow}
\usepackage{placeins}
\usepackage{pgf}
\usepackage{tikz}
\usepackage{xcolor}
\usepackage{mathtools}
\usepackage{algorithmicx}
\usepackage{algpseudocode}
\usepackage{dsfont}
\usepackage{stmaryrd}
\usepackage{cite}		

\usepackage{hyperref}
\hypersetup{
    colorlinks,
    linkcolor={black},
    citecolor={black},
    urlcolor={black}
}

\usepackage{cleveref}

\theoremstyle{definition}
\newtheorem{theorem}{Theorem}[section]
\newtheorem{lemma}[theorem]{Lemma}
\newtheorem{proposition}[theorem]{Proposition}
\newtheorem{corollary}[theorem]{Corollary}

\newtheorem{remark}[theorem]{Remark}

\theoremstyle{remark}

\crefname{equation}{}{}
\crefname{theorem}{Theorem}{Theorems}
\crefname{lemma}{Lemma}{Lemmas}
\crefname{corollary}{Corollary}{Corollaries}
\crefname{proposition}{Proposition}{Propositions}
\crefname{remark}{Remark}{Remarks}
\crefname{section}{Section}{Sections}
\crefname{figure}{Figure}{Figures}

\DeclarePairedDelimiter{\abs}{\lvert}{\rvert}
\DeclarePairedDelimiter{\norm}{\lVert}{\rVert}

\DeclareMathOperator{\vspan}{span} 	
\DeclareMathOperator{\Real}{Re}
\DeclareMathOperator{\Imag}{Im}

\DeclareMathOperator*{\argmin}{argmin}
\newcommand{\conj}[1]{\overline{#1}}		

\newcommand{\C}{\mathbb{C}}
\newcommand{\R}{\mathbb{R}}

\newcommand{\de}{\,\text{d}}	
\newcommand{\dde}{\text{d}}		

\newcommand{\sol}{\texttt{x}}		
\newcommand{\err}{\texttt{err}}		
\newcommand{\res}{\texttt{res}}		
\newcommand{\f}{\texttt{f}}			
\newcommand{\qf}{\texttt{q}}		

\renewcommand{\vec}[1]{\boldsymbol{#1}}		

\newcommand{\kryl}{\mathcal{K}}		
\newcommand{\rat}{\mathcal{Q}}		
\newcommand{\poly}{\Pi}				

\graphicspath{{./}}

\makeatletter
\renewcommand*\env@matrix[1][*\c@MaxMatrixCols c]{%
  \hskip -\arraycolsep
  \let\@ifnextchar\new@ifnextchar
  \array{#1}}
\makeatother

\numberwithin{equation}{section}

\begin{document}

\title[Rational Krylov error bounds for matrix functions]{Error bounds for the approximation of matrix functions with rational Krylov methods}

\author{Igor Simunec}
\address{Scuola Normale Superiore, Piazza dei Cavalieri 7, 56126 Pisa, Italy}
\email{igor.simunec@sns.it}

\begin{abstract}
	We obtain an expression for the error in the approximation of $f(A) \vec b$ and $\vec b^T f(A) \vec b$ with rational Krylov methods, where $A$ is a symmetric matrix, $\vec b$ is a vector and the function $f$ admits an integral representation. 
	The error expression is obtained by linking the matrix function error with the error in the approximate solution of shifted linear systems using the same rational Krylov subspace, and it can be exploited to derive both a priori and a posteriori error bounds. 
	The error bounds are a generalization of the ones given in [T.~Chen, A.~Greenbaum, C.~Musco, C.~Musco, SIAM J.~Matrix Anal.~Appl., 43 (2022), pp.~787--811] for the Lanczos method for matrix functions. 
	A technique that we employ in the rational Krylov context can also be applied to refine the bounds for the Lanczos case.
\end{abstract}


\maketitle

\section{Introduction}

An important problem in numerical linear algebra is the computation of the action of a matrix function $f(A)$ \cite{Higham08-book} on a vector $\vec b$, where the matrix $A$ is usually large and sparse. The matrix-vector product $f(A) \vec b$ can be efficiently approximated via projection on polynomial \cite{Saad92, EFLSV02, FGS14} or rational Krylov subspaces \cite{Ruhe94, DruskinKnizhnerman98, MoretNovati04, GuettelKnizhnerman13, Guettel13, BenziSimunec22}, without resorting to the expensive computation of the whole matrix function $f(A)$. Each iteration of a polynomial or rational Krylov methods requires, respectively, the computation of a matrix-vector product or the solution of a shifted linear system with the matrix $A$. 

In applications it is often essential to have bounds or estimates on the error of the Krylov approximation, in order to either predict the number of iterations needed to reach a given accuracy or to have a reliable stopping criterion. In the literature there have been several papers devoted to developing a priori and a posteriori error bounds and estimates, both in the polynomial \cite{FrommerSimoncini08, FrommerSchweitzer16, CGMM21} and in the rational Krylov setting \cite{BeckermannReichel09, MoretNovati19, MasseiRobol21}.
In particular, the authors of \cite{CGMM21} analyze the error of the Lanczos method for functions of Hermitian matrices, and by means of the Cauchy integral formula they derive an expression for the error that can be used to obtain both a priori and a posteriori error bounds.

In this work we follow a similar approach to analyze the error in the approximation of $f(A) \vec b$ by means of a rational Krylov method, and we derive both a priori and a posteriori bounds in the case of a symmetric matrix $A$, for functions that have a certain integral representation. The results that we obtain can be interpreted as a generalization to the rational Krylov case of the ones presented in~\cite{CGMM21}.

The rest of the paper is organized as follows. In \cref{subsec:basic-definitions} we introduce basic definitions regarding Krylov subspaces and some notation that is used throughout the paper. 
In \cref{sec:rational-krylov--linear-sytems} we establish properties of the residuals of shifted linear systems solved with a rational Krylov method. 
These properties are used in \cref{sec:integral-error-expression} to derive integral expressions of the matrix function error in terms of errors or residuals of shifted linear systems, which are in turn used to obtain a priori and a posteriori error bounds. 
These bounds are discussed in detail, respectively, in \cref{subsec:a-priori-bounds} and \cref{subsec:a-posteriori-bounds}. 
In \cref{sec:linear-system-bounds} we include an bound on the residual of a linear system solved with a rational Krylov method, which is required for the a priori bounds on the matrix function error. 
In \cref{sec:quadratic-forms} we generalize our approach to quadratic forms $\vec b^T f(A) \vec b$, and in \cref{sec:numerical-experiments} we illustrate the bounds with some numerical experiments. \cref{sec:conclusions} contains some concluding remarks.

\subsection{Basic definitions}
\label{subsec:basic-definitions}

Given a matrix $A \in \R^{n \times n}$ and a vector $\vec b \in \R^n$, a polynomial Krylov subspace is defined by
\begin{equation*}
	\kryl_m(A,\vec b) = \vspan\{ \vec b, A \vec b, \dots, A^{m-1} \vec b\} = \{ p(A) \vec b : p \in \poly_{m-1}\},
\end{equation*}
where $\poly_{m-1}$ denotes the set of polynomials of degree $\le m-1$.
Given a sequence of poles~$\{ \xi_j \}_{j \ge 1} \subset (\C\cup \{\infty\}) \setminus {\sigma(A)}$, where $\sigma(A)$ denotes the spectrum of $A$, we can define the denominator polynomial $q_{m-1}(z) = \prod_{j = 1}^{m-1}(1 - z/\xi_j)$ and the associated rational Krylov subspace
\begin{equation*}
	\rat_m(A,\vec b) = q_{m-1}(A)^{-1} \kryl_m(A, \vec b) = \Big\{ r(A) \vec b : r(z) = \frac{p_{m-1}(z)}{q_{m-1}(z)}, \text{with } p_{m-1} \in \poly_{m-1}\Big\}.
\end{equation*}
Note that $\kryl_m(A, \vec b)$ is a special case of a rational Krylov subspace, obtained by taking all poles $\xi_j = \infty$.

The subspaces $\rat_j(A, \vec b)$ for $j = 1, \dots, m$ form a nested sequence of strictly increasing dimension, provided that $m$ is smaller than the invariance index $M$, i.e.~the smallest integer such that $\kryl_M(A, \vec b) = \kryl_{M+1}(A, \vec b)$. For simplicity, we are going to assume that this is always the case. An orthonormal basis $V_m = [\vec v_1, \dots, \vec v_m]$ of the subspace $\rat_m(A, \vec b)$ can be constructed iteratively using the rational Arnoldi algorithm, introduced by Ruhe in~\cite{Ruhe94}. 

The rational Arnoldi algorithm also outputs two upper Hessenberg matrices $\underline {H_m}$ and $\underline{K_m}$ of size $(m+1) \times m$ that satisfy the rational Arnoldi relation 
\begin{equation*}
	A V_{m+1} \underline{K_m} = V_{m+1} \underline{H_m}.
\end{equation*}
We will denote by $H_m$ and $K_m$ the leading principal $m \times m$ blocks of $\underline{H_m}$ and $\underline{K_m}$. 

If $\xi_m = \infty$, then the last row of $\underline{K_m}$ is zero and $K_m$ is invertible \cite[Section~3.1]{Guettel13}, so we can rewrite the rational Arnoldi relation as
\begin{equation*}
	A V_{m} K_m = V_{m+1} \underline{H_m}
\end{equation*}
and we have
\begin{equation*}
	A_m := V_m^T A V_m = H_m K_m^{-1}.
\end{equation*}
The matrix-vector product $f(A) \vec b$ can be approximated with
\begin{equation}
	\label{eqn:rational-krylov-approximation}
	\f_m := V_m f(A_m) V_m^T \vec b = V_m f(A_m) \vec e_1 \norm{\vec b}_2.
\end{equation}
We refer the reader to \cite{GuettelThesis, Guettel13} for a more thorough discussion of rational Krylov methods for the computation of matrix functions, and to \cite{BerljafaGuettel15} for an analysis of rational Krylov decompositions.

We consider the class of functions $f$ that admit the integral expression
\begin{equation}
	\label{eqn:integral-representation-f}
	f(x) = \int_\Gamma (x - z)^{-1} \de \mu(z),
\end{equation}
for a suitable contour $\Gamma \subset \C$ and measure $\dde \mu(z)$. The representation \cref{eqn:integral-representation-f} reduces to the Cauchy integral formula if $\de \mu(z) = - \dfrac{1}{2 \pi i} f(z) \de z$ and $\Gamma$ is a simple closed curve that encloses $x$, and to the Cauchy-Stieltjes integral representation for $\Gamma = (-\infty, 0]$ and a positive measure~$\dde\mu(z)$ (see e.g.~\cite[Definition~2]{MasseiRobol21}).
The following are some examples of Cauchy-Stieltjes functions \cite[Examples 1.3 and 1.4]{BeckermannReichel09}:
\begin{equation}
	\label{eqn:cauchy-stieltjes-examples}
	\begin{aligned}
		x^{-\alpha} &= \frac{\sin(\alpha \pi)}{\pi}\int_{-\infty}^0 \frac{(-z)^{-\alpha}}{x - z} \de z, \qquad \alpha \in (0,1), \\
		\frac{\log(1+x)}{x} &= - \int_{-\infty}^{-1} \frac{z^{-1}}{x - z} \de z.
	\end{aligned}
\end{equation}

Using the integral representation~\cref{eqn:integral-representation-f} for the function $f$, we can write the error in the rational Krylov approximation of $f(A) \vec b$ in the form
\begin{equation}
	\label{eqn:basic-integral-error-ratkrylov--matvec}
	f(A) \vec b - \f_m = \int_\Gamma \big((A - z I)^{-1} \vec b - V_m (A_m - z I)^{-1} V_m^T \vec b \big) \de \mu(z) 
	= \int_\Gamma \err_m(z) \de \mu(z),
\end{equation}
where $\err_m(z)$ denotes the error in the solution of the shifted linear system $(A - z I) \vec x = \vec b$ using the rational Krylov subspace $\rat_m(A, \vec b)$. Motivated by \cref{eqn:basic-integral-error-ratkrylov--matvec}, in \cref{sec:rational-krylov--linear-sytems} we analyze the solution of shifted linear systems with rational Krylov methods, starting by recalling the polynomial Krylov case.

\section{Results on rational Krylov methods for shifted linear systems}
\label{sec:rational-krylov--linear-sytems}

\subsection{Background: solution of linear systems with FOM}
\label{subsec:linear-systems-fom}

Let us start by recalling some basic facts regarding the solution of shifted linear systems using the full orthogonalization method (FOM) \cite{Saad03}, which will simplify the presentation of the next part.

Recall that after $m$ iterations FOM constructs a basis $V_m$ of the polynomial Krylov subspace $\kryl_m(A, \vec b)$ and approximates the solution to the linear system $A \vec x = \vec b$ with
\begin{equation*}
	\vec x_m = V_m A_m^{-1} \vec e_1 \norm{\vec b}_2, \qquad \text{where } A_m = V_m^T A V_m.
\end{equation*}
This is equivalent to imposing that the residual $\vec r_m = \vec b - A \vec x_m$ is orthogonal to the basis~$V_m$. 

The residual $\vec r_m$ can be elegantly expressed in terms of the characteristic polynomial of~$A_m$. If we denote by $\chi_m(z)$ the characteristic polynomial of $A_m$, that is $\chi_m(z) = \det(z I - A_m)$, by \cite[eq.~(3.8)]{PPV95} we have 
\begin{equation}
	\label{eqn:conjugate-gradient-residual-charpoly}
	\vec r_m = \chi_m(A) \vec b / \chi_m(0).
\end{equation}
Moreover, since $\vec r_m \in \kryl_{m+1}(A, \vec b)$ and $\vec r_m \perp \kryl_m(A, \vec b)$, the residual $\vec r_m$ is proportional to the next basis vector $\vec v_{m+1}$ (see, e.g.~\cite[Proposition~6.7]{Saad03}).

A similar argument can be used for the solution of a shifted linear system $(A - t I) \vec x(t) = \vec b$, for any $t \in \R$. Due to the shift invariance of polynomial Krylov subspaces, i.e.~$\kryl_m(A - t I, \vec b) = \kryl_m(A, \vec b)$, after~$m$ iterations of FOM applied to the shifted system we have constructed the same basis $V_m$ as in the case of the linear system $A \vec x = \vec b$, and therefore the approximate solution is given by
\begin{equation*}
	\vec x_m(t) = V_m (A_m^t)^{-1} \vec e_1 \norm{\vec b}_2, \qquad \text{where } A_m^t := V_m^T (A - t I) V_m = A_m - t I.
\end{equation*}
The residual can be expressed using \eqref{eqn:conjugate-gradient-residual-charpoly} with $A_m^t$ in place of $A_m$, yielding
\begin{equation*}
	\vec r_m(t) = \vec b - (A - t I) \vec x_m(t) = \chi_m^t(A - t I) \vec b / \chi_m^t(0),
\end{equation*}
where 
\begin{equation*}
	\chi_m^t(z) := \det(z I - A_m^t) = \chi_m(z + t).
\end{equation*}
With simple algebraic manipulations, the above expression can be rewritten as
\begin{equation}
	\label{eqn:conjugate-gradient-shifted-residual-charpoly}
	\vec r_m(t) = \chi_m(A) \vec b / \chi_m(t) = \frac{\chi_m(0)}{\chi_m(t)} \vec r_m(0).
\end{equation}

In other words, the residuals of shifted linear systems are all collinear, and they are proportional to $\vec v_{m+1}$, the next vector in the Krylov basis. The constant in \eqref{eqn:conjugate-gradient-shifted-residual-charpoly} can be written explicitly in terms of the eigenvalues of the projected matrix $A_m$. These facts are well known in the literature, see for instance~\cite[Proposition~2.1]{Simoncini03} and \cite[Lemma~5]{EFLSV02}.

\subsection{Rational Arnoldi decompositions}

In order to generalize the results of \cref{subsec:linear-systems-fom} to rational Krylov subspaces, we start by recalling some properties of rational Arnoldi decompositions.
Let us consider the rational Arnoldi decomposition
\begin{equation*}
	A V_{m+1} \underline{K_m} = V_{m+1} \underline{H_m}
\end{equation*}
obtained after $m$ iterations of the rational Arnoldi algorithm. We can write this relation more explicitly as
\begin{equation}
	\label{eqn:rad-explicit}
	A V_m K_m + A \vec v_{m+1} k_{m+1, m} \vec e_m^T = V_m H_m + \vec v_{m+1} h_{m+1, m} \vec e_m^T,
\end{equation}
where $h_{m+1, m}$ and $k_{m+1, m}$ are the last subdiagonal elements of $\underline{H_m}$ and $\underline{K_m}$, respectively.    
The ratios of corresponding subdiagonal elements of $\underline{H_m}$ and $\underline{K_m}$ are the poles ${\xi_1, \dots, \xi_m}$ of the rational Krylov subspace. In particular, $h_{m+1, m}/k_{m+1, m} = \xi_m$, so $k_{m+1, m} = 0$ when $\xi_m = \infty$. 
We refer to \cite[Section~5.1]{GuettelThesis} or \cite{BerljafaGuettel15} for more details. 

Assuming that $K_m$ is nonsingular, for all $z \in \C$ we have
\begin{equation}
	\label{eqn:shifted-rad-explicit-invertible-Km}
	(A -z I) V_m = V_m (H_m K_m^{-1} - z I) + (h_{m+1, m}I - k_{m+1, m} A )\vec v_{m+1} \vec e_m^T K_m^{-1}.
\end{equation}
Defining $A_m = V_m^T A V_m$, we see from \cref{eqn:rad-explicit} that 
\begin{equation*}
	A_m = H_m K_m^{-1} - k_{m+1, m} V_m^T A \vec v_{m+1} \vec e_m^T K_m^{-1},
\end{equation*}
and combining with \cref{eqn:shifted-rad-explicit-invertible-Km} we get for $A_m$ the identity
\begin{align}
	\label{eqn:shifted-rad-explicit-Am}
	\nonumber
	(A - zI) V_m &= V_m (A_m - z I) + \left(h_{m+1, m}I - k_{m+1, m} (I - V_m V_m^T) A \right) \vec v_{m+1} \vec e_m^T K_m^{-1} \\
	&= V_m (A_m - z I) + (I - V_m V_m^T) (h_{m+1, m}I - k_{m+1, m} A ) \vec v_{m+1} \vec e_m^T K_m^{-1},
\end{align}
where we used the fact that $\vec v_{m+1} \perp \vspan(V_m)$ for the last equality.

\subsection{Solution of shifted linear systems}

The results that we derive in this section have strong similarities with the ones presented in \cref{subsec:linear-systems-fom} for FOM. Consider the family of shifted linear system
\begin{equation*}
	(A - z I) \vec x = \vec b, \qquad z \in \C.
\end{equation*}
In analogy to FOM and to \cref{eqn:rational-krylov-approximation}, we can extract from the rational Krylov subspace $\rat_m(A, \vec b)$ the approximate solution
\begin{equation}
	\label{eqn:shifted-linsys-solution-ratkryl}
	\sol_m(z) = V_m (A_m - z I)^{-1} \vec e_1 \norm{\vec b}_2.
\end{equation}
Let us introduce the notation 
\begin{align*}
	\res_m(z) &= \vec b - (A - z I) \sol_m(z) \\
	\err_m(z) &= (A - z I)^{-1} \res_m(z)
\end{align*}
for the linear system residual and error, respectively.
Using \cref{eqn:shifted-rad-explicit-Am}, the residual $\res_m(z)$ can be written as
\begin{equation}
	\label{eqn:shifted-residuals-collinearity}
	\res_m(z) = -(I - V_m V_m^T) (h_{m+1, m}I - k_{m+1, m} A ) \vec v_{m+1} \vec e_m^T K_m^{-1}(A_m - tI)^{-1} \vec e_1 \norm{\vec b}_2.
\end{equation}
This shows that the residuals are collinear for all $z \in \C$. In particular, when $\xi_m = \infty$ we have $k_{m+1, m} = 0$ and \cref{eqn:shifted-residuals-collinearity} simplifies to
\begin{equation}
	\label{eqn:shifted-residuals-collinearity-infty}
	\res_m(z) = - h_{m+1, m} \vec v_{m+1} \vec e_m^T K_m^{-1}(A_m - z I)^{-1} \vec e_1 \norm{\vec b}_2,
\end{equation}
i.e.~the residuals are collinear to the next basis vector $\vec v_{m+1}$, as in the polynomial case (see \cref{subsec:linear-systems-fom}). The identity \cref{eqn:shifted-residuals-collinearity-infty} has been used in past literature, see for instance \cite[Section~6.6.2]{GuettelThesis} and \cite[Section~4.3]{BRS22}.

For simplicity, let us focus first on the case $z = 0$, and let us use the notation $\res_m(0) = \res_m$ and similarly $\sol_m(0) = \sol_m$. 
The residual $\res_m$ is orthogonal to $\rat_m(A, \vec b)$ because of the factor $(I - V_m V_m^T)$ in \cref{eqn:shifted-residuals-collinearity}; recalling that $\xi_m = h_{m+1, m}/k_{m+1, m}$, we also see that $\res_m$ belongs to the subspace
\begin{equation*}
	(I - \xi_m^{-1} A) \rat_{m+1}(A, \vec b) = q_{m-1}(A)^{-1} \kryl_{m+1}(A, \vec b) \supset \rat_m(A, \vec b).
\end{equation*}
If we let $\chi_m$ denote the characteristic polynomial of $A_m$, that is $\chi_m(z) = \det(z I - A_m)$, by~\cite[Lemma~4.5]{GuettelThesis} we have
\begin{equation*}
	\chi_m(A) q_{m-1}(A)^{-1} \vec b \perp \rat_m(A, \vec b),
\end{equation*}
and moreover $\chi_m(A) q_{m-1}(A)^{-1} \vec b \in q_{m-1}(A)^{-1} \kryl_{m+1}(A, \vec b)$. The vectors $\chi_m(A) q_{m-1}(A)^{-1} \vec b$ and $\vec r_m$ belong to the same subspace of dimension $m+1$ which contains $\rat_m(A, \vec b)$ and they are both orthogonal to $\rat_m(A, \vec b)$, so they must be collinear, i.e.~there exists a constant $\alpha \in \R$ such that 
\begin{equation*}
	\res_m = \alpha \chi_m(A) q_{m-1}(A)^{-1} \vec b.
\end{equation*}
The value of the constant $\alpha$ can be determined by looking at the constant terms in the identity
\begin{equation*}
	q_{m-1}(A) \vec b - q_{m-1}(A) A \sol_m = \alpha \chi_m(A) \vec b,
\end{equation*}
which gives us
\begin{equation*}
	q_{m-1}(0) = \alpha \chi_m(0).
\end{equation*}
As a consequence, we have the following elegant expression for the residual with $z = 0$,
\begin{equation}
	\label{eqn:ratkryl-linsys-residual-charpoly}
	\res_m(0) = \frac{q_{m-1}(0)}{\chi_m(0)} \chi_m(A) q_{m-1}(A)^{-1} \vec b.
\end{equation}

This result can now be easily generalized to all $z \in \C$. Consider the shifted linear system $(A - z I) \vec x = \vec b$, which has the approximate solution
\begin{equation*}
	\sol_m(z) = V_m (A_m^z)^{-1} \vec e_1 \norm{\vec b}_2, \qquad \text{where } A_m^z = V_m^T(A - z I) V_m = A_m - z I.
\end{equation*}
Observe that the subspace $\rat_m(A, \vec b)$ with poles $\{\xi_1, \dots, \xi_m\}$ coincides with the subspace $\rat_m(A - z I, \vec b)$ defined using the shifted poles $\{ \xi_1 - z, \, \dots, \, \xi_m - z \}$. The denominator polynomial associated to this rational Krylov subspace is therefore
\begin{equation*}
	q_{m-1}^z (z) = \prod_{j = 1}^{m-1} (x - (\xi_j - z)) = q_{m-1}(x + z),
\end{equation*}
and the characteristic polynomial $\chi_m^z$ of the projected matrix $A_m^z$ is 
\begin{equation*}
	\chi_m^z(x) = \det(x I - A_m^z) = \chi_m(x + z).
\end{equation*}
If we write the residual $\res_m(z) = \vec b - (A - z I) \sol_m(z)$ using \eqref{eqn:ratkryl-linsys-residual-charpoly} with $A$ replaced by $A - z I$, we get
\begin{align}
	\nonumber
	\res_m(z) &= \frac{q_{m-1}^z(0)}{\chi_m^z(0)} \chi_m^z(A - z I) q_{m-1}^z(A - z I)^{-1} \vec b \\
	\nonumber
	&= \frac{q_{m-1}(z)}{\chi_m(z)} \chi_m(A) q_{m-1}(A)^{-1} \vec b \\
	\label{eqn:ratkryl-shifted-linsys-residual-charpoly}
	&= \frac{q_{m-1}(z)}{q_{m-1}(0)} \cdot \frac{\chi_m(0)}{\chi_m(z)} \, \res_m (0).
\end{align}
If we let $\theta_j^{(m)}$, $j = 1, \dots, m$, denote the eigenvalues of $A_m$, we obtain the following more explicit expression for the shifted residual:
\begin{equation*}
	\res_m(z) = \varphi_m(z) \res_m (0), \qquad \text{where } \varphi_m(z) = \prod_{j = 1}^{m-1} \frac{\xi_j - z}{\xi_j} \prod_{j = 1}^m \frac{\theta_j^{(m)}}{\theta_j^{(m)} - z}.
\end{equation*}
Given $w \in \C$, it is easy to derive from \cref{eqn:ratkryl-shifted-linsys-residual-charpoly} the identity
\begin{equation}
	\label{eqn:linsys-residual-relation-z-w}
	\res_m (z) = \frac{q_{m-1}(z)}{q_{m-1}(w)} \cdot \frac{\chi_m(w)}{\chi_m(z)} \, \res_m (w).
\end{equation} 
Furthermore, comparing \cref{eqn:linsys-residual-relation-z-w} with \cref{eqn:shifted-residuals-collinearity}, we also obtain
\begin{equation*}
	\frac{q_{m-1}(z)}{q_{m-1}(w)} \cdot \frac{\chi_m(w)}{\chi_m(z)} = \frac{\vec e_m^T K_m^{-1} (A_m - z I)^{-1} \vec e_1}{\vec e_m^T K_m^{-1} (A_m - w I)^{-1} \vec e_1}.
\end{equation*}
We are going to use the previous expressions in \cref{sec:integral-error-expression} to obtain bounds for the error in the approximation of $f(A) \vec b$.

\section{Bounds for the matrix function error}
\label{sec:integral-error-expression}

Recalling \cref{eqn:basic-integral-error-ratkrylov--matvec}, the error in the approximation of $f(A) \vec b$ can be written as
\begin{equation*}
	f(A) \vec b - \f_m = \int_\Gamma \err_m(z) \de \mu(z) = \int_\Gamma (A - z I)^{-1} \res_m(z) \de \mu(z).
\end{equation*}
Using~\eqref{eqn:linsys-residual-relation-z-w} we can write $\res_m(z)$ in terms of $\res_m(w)$ for any $z$, $w \in \C$.
A similar relation also holds for $\err_m(z)$ and $\err_m(w)$, indeed we have
\begin{equation}
	\label{eqn:linsys-error-relation-z-w}
	\begin{aligned}
		\err_m(z) &= (A - z I)^{-1} \res_m(z) \\
		&= \frac{q_{m-1}(z)}{q_{m-1}(w)} \cdot \frac{\chi_m(w)}{\chi_m(z)} (A - z I)^{-1} \res_m(w) \\
		&= \frac{q_{m-1}(z)}{q_{m-1}(w)} \det \big(h_{w, z}(A_m)\big) h_{w, z}(A) \err_m(w),
	\end{aligned}
\end{equation}
where we used the notation $h_{w, z}(t) := \dfrac{t - w}{t - z}$, borrowed from \cite{CGMM21}. We are also going to use the notation $k_z(t) := (t - z)^{-1}$. 

The identities \cref{eqn:linsys-residual-relation-z-w,eqn:linsys-error-relation-z-w} can be used to prove the following theorem, which gives us an expression for the error in the approximation of $f(A) \vec b$ in terms of the error or residual of shifted linear systems. 

\begin{theorem}
	\label{thm:integral-error-ratkrylov--matvec}
	Assume that $f$ has the integral representation \cref{eqn:integral-representation-f}, and denote by $\f_m$ the approximation to $f(A) \vec b$ after $m$ iterations of a rational Krylov method with denominator polynomial $q_{m-1}$. Let $\mathcal{D} = \{\xi_1, \dots, \xi_{m-1} \} \cup \sigma(A) \cup \sigma(A_m)$ and let $w$ be a function $w : \Gamma \to \C \setminus \mathcal{D}$. We have
	\begin{equation}
		\label{eqn:integral-error-ratkrylov--matvec}
		\begin{aligned}
			f(A) \vec b - \f_m &= \int_\Gamma \frac{q_{m-1}(z)}{q_{m-1}(w(z))} \det \big( h_{w(z), z}(A_m) \big) h_{w(z), z}(A) \err_m(w(z)) \de \mu(z) \\
			&= \int_\Gamma \frac{q_{m-1}(z)}{q_{m-1}(w(z))} \det \big( h_{w(z), z}(A_m) \big) \, k_z(A) \, \res_m(w(z)) \de \mu(z),
		\end{aligned}
	\end{equation}
\end{theorem}
\begin{proof}
	In \cref{eqn:basic-integral-error-ratkrylov--matvec}, for each $z \in \Gamma$ replace $\err_m(z)$ with the expressions in \cref{eqn:linsys-error-relation-z-w} using $w = w(z)$ to obtain the first identity in \cref{eqn:integral-error-ratkrylov--matvec}. The second identity is then easily obtained by recalling the definition of $k_z$.
\end{proof}

\begin{remark}
	\label{rem:comparison-with-lanczos-case}
	An identity similar to the one in \cref{thm:integral-error-ratkrylov--matvec} has been given in \cite[Corollary~2.5]{CGMM21} for the error of the approximation of $f(A) \vec b$ with the Lanczos method, using the Cauchy integral representation of $f$. 
	Indeed, \cref{thm:integral-error-ratkrylov--matvec} reduces to \cite[Corollary~2.5]{CGMM21} by taking $q_{m-1}(z) \equiv 1$, $\dde \mu(z) = -\dfrac{1}{2 \pi i} f(z) \de z$ and a constant $w(z) = z$.
	Apart from the factors involving the denominator $q_{m-1}$ of the rational Krylov subspace, an important difference from the result in \cite{CGMM21} is the use of a function $w(z)$ instead of a constant parameter $w$. This modification, which might appear minor at a quick glance, turns out to be a crucial element for dealing with the rational Krylov case. 
	We are going to thoroughly discuss this point in \cref{subsec:bound-discussion}.
	
\end{remark}

Although it is not practically feasible to evaluate the error expression in \cref{eqn:integral-error-ratkrylov--matvec} directly since it depends explicitly on the matrix $A$, it can be employed to derive bounds for the error that are easily computable. 
The following corollary is a simple consequence of \cref{thm:integral-error-ratkrylov--matvec}, and it is the analogue of \cite[Theorem~2.6]{CGMM21} for rational Krylov instead of Lanczos.
We use the notation $\norm{g}_X = \max_{x \in X} \abs{g}$ for the supremum norm of the function $g$ on the set $X \subset \C$, and $\norm{A}_2$ for the spectral norm of the matrix $A$. We denote the eigenvalues of $A_m$ by $\lambda_i(A_m)$, $i = 1, \dots, m$.  

\begin{corollary}
	\label{cor:integral-error-bound-ratkrylov--general}
	With the same hypothesis of \cref{thm:integral-error-ratkrylov--matvec}, assume further that $A$ is symmetric and that for some $\mathcal{S}_0, \mathcal{S}_1, \dots, \mathcal{S}_m \subset \R$ we have $\sigma(A) \subset \mathcal{S}_0$, and $\lambda_i(A_m) \in \mathcal{S}_i$ for all~$i = 1, \dots, m$. Then we have the inequalities
	\begin{equation*}
		\norm{f(A) \vec b - \f_m}_2 \le \int_\Gamma \abs*{ \frac{q_{m-1}(z)}{q_{m-1}(w(z))}} \prod_{j = 0}^m \norm{h_{w(z), z}}_{\mathcal{S}_j} \cdot \norm{\err_m(w(z))}_2 \, \abs{\dde \mu(z)}
	\end{equation*}
	and
	\begin{equation*}
		\norm{f(A) \vec b - \f_m}_2 \le \int_\Gamma \abs*{ \frac{q_{m-1}(z)}{q_{m-1}(w(z))}} \prod_{j = 1}^m \norm{h_{w(z), z}}_{\mathcal{S}_j} \cdot \norm{k_z}_{\mathcal{S}_0} \cdot \norm{\res(w(z))}_2 \, \abs{\dde \mu(z)}.
	\end{equation*}
\end{corollary}
\begin{proof}
	Since the eigenvalues of $A$ are contained in $\mathcal{S}_0$, we have \begin{equation*}
		\norm{h_{w(z), z}(A)}_2 = \max_{\lambda \in \sigma(A)} \abs{h_{w(z), z}(\lambda)} \le \norm{h_{w(z), z}}_{\mathcal{S}_0}.
	\end{equation*}
	Similarly, we have
	\begin{equation*}
		\abs*{\det \big(h_{w(z), z}(A_m)\big) } = \prod_{j = 1}^m \abs* {h_{w(z), z}(\lambda_j(A_m))} \le \prod_{j = 1}^m \norm{h_{w(z), z}}_{\mathcal{S}_j}.
	\end{equation*}
	By applying the above inequalities to the first identity in \cref{thm:integral-error-ratkrylov--matvec}, we obtain
	\begin{align*}
		\norm{f(A) \vec b - \f_m}_2 &\le \int_\Gamma  \abs*{ \frac{q_{m-1}(z)}{q_{m-1}(w(z))}} \cdot \abs{\det(h_{w(z), z}(A_m))} \cdot \norm{h_{w(z), z}(A)}_{2}  \norm{\err_m(w(z))}_2 \, \abs{\dde \mu(z)} \\
		&\le  \int_\Gamma \abs*{ \frac{q_{m-1}(z)}{q_{m-1}(w(z))}} \cdot \norm{h_{w(z), z}}_{\mathcal{S}_0} \cdot \prod_{j = 1}^m \norm{h_{w(z), z}}_{\mathcal{S}_j}  \cdot \norm{\err_m(w(z))}_2 \, \abs{\dde \mu(z)}.
	\end{align*}
	The second inequality follows with the same argument, using the error expression in terms of $\res_m(w(z))$ from \cref{thm:integral-error-ratkrylov--matvec}.
\end{proof}

Observe that if we use a constant function $w(z) \equiv w$ in \cref{cor:integral-error-bound-ratkrylov--general} the bounds are simplified, since the terms $\norm{\err_m(w)}_2$ and $\norm{\res_m(w)}_2$ no longer depend on $z$ and thus can be bounded independently from the integral term. 

\cref{cor:integral-error-bound-ratkrylov--general} can be interpreted as either an a priori or an a posteriori bound, depending on the choices for the sets $\mathcal{S}_j$. 
For example, we can easily obtain an a priori bound by taking $\mathcal{S}_j = [a, b] \supset \sigma(A)$ for $j = 0, 1, \dots, m$.
On the other hand, once $A_m$ has been computed and its eigenvalues are known, we can obtain an a posteriori bound by taking $\mathcal{S}_0 = [a, b]$ and $\mathcal{S}_j = \{ \lambda_j(A_m) \}$, for $j = 1, \dots, m$. The two following corollaries are restatements of \cref{cor:integral-error-bound-ratkrylov--general} in the a priori and a posteriori setting, respectively, using the choices described above for the sets $\mathcal{S}_j$, $j = 0, \dots, m$.

\begin{corollary}
	\label{cor:integral-error-bound--a-priori}
	Under the assumptions of \cref{thm:integral-error-ratkrylov--matvec}, if $A$ is symmetric, we have the a priori bounds
	\begin{equation*}
		\norm{f(A) \vec b - \f_m}_2 \le \int_\Gamma \abs*{ \frac{q_{m-1}(z)}{q_{m-1}(w(z))}} \cdot \norm{h_{w(z), z}}_{[a, b]}^{m+1} \cdot \norm{\err_m(w(z))}_2 \, \abs{\dde \mu(z)}
	\end{equation*}
	and
	\begin{equation*}
		\norm{f(A) \vec b - \f_m}_2 \le \int_\Gamma \abs*{ \frac{q_{m-1}(z)}{q_{m-1}(w(z))}} \cdot \norm{h_{w(z), z}}_{[a, b]}^{m} \cdot \norm{k_z}_{[a, b]} \cdot \norm{\res_m(w(z))}_2 \, \abs{\dde \mu(z)}.
	\end{equation*}	
\end{corollary}
\begin{corollary}
	\label{cor:integral-error-bound--a-posteriori}
	Under the assumptions of \cref{thm:integral-error-ratkrylov--matvec}, if $A$ is symmetric, we have the a posteriori bounds
	\begin{equation*}
		\norm{f(A) \vec b - \f_m}_2 \le \int_\Gamma \abs*{ \frac{q_{m-1}(z)}{q_{m-1}(w(z))} \cdot \frac{\chi_m(w(z))}{\chi_m(z)}}\cdot \norm{h_{w(z), z}}_{[a, b]} \cdot \norm{\err_m(w(z))}_2 \, \abs{\dde \mu(z)}
	\end{equation*}
	and
	\begin{equation*}
		\norm{f(A) \vec b - \f_m}_2 \le \int_\Gamma \abs*{ \frac{q_{m-1}(z)}{q_{m-1}(w(z))} \cdot \frac{\chi_m(w(z))}{\chi_m(z)}}\cdot \norm{k_z}_{[a, b]} \cdot \norm{\res_m(w(z))}_2 \, \abs{\dde \mu(z)}.
	\end{equation*}
\end{corollary}
	
If additional a priori information on the spectrum of $A$ is available, it can be incorporated in the sets $\mathcal{S}_j$ in order to get more precise bounds. For example, assume that in addition to the spectral interval $[a, b]$ we know that all except for $10$ eigenvalues of $A$ are contained in the smaller interval $[a, c]$, with $a < c < b$. Then, by the Cauchy interlacing theorem \cite[Theorem~8.1.7]{GolubVanLoan13-book} we can conclude that at most $10$ eigenvalues of $A_m$ are not inside the interval $[a, c]$, so we can take $\mathcal{S}_j = [a, b]$ for $j = 1, \dots, 10$ and $\mathcal{S}_j = [a, c]$ for $j = 11, \dots, m$. This kind of choice for the sets $\mathcal{S}_j$ exploits more spectral information of the matrix compared to the bounds in \cref{cor:integral-error-bound--a-priori}, so it would result in more refined bounds that are able to better capture the convergence behavior of rational Krylov methods.

\subsection{Bound discussion}
\label{subsec:bound-discussion}

The error bounds in \cref{cor:integral-error-bound--a-priori,cor:integral-error-bound--a-posteriori} depend on several parameters, such as the curve $\Gamma$ and the function $w(z)$, and on the choice between $\norm{\err_m(w(z))}_2$ and $\norm{\res_m(w(z))}_2$.

Since the factors $\norm{h_{w, z}}_{[a, b]}$ and $\norm{k_z}_{[a, b]}$ diverge when $z$ approaches the interval $[a, b]$, when using the Cauchy integral formula it is a good idea to take $\Gamma$ as far as possible from the spectrum of $A$. For example, for functions such as $f(z) = z$ with a growth of $o(\abs{z})$ for $\abs{z} \to \infty$, we can choose $\Gamma$ as a keyhole contour around the negative real line, which reduces to $\Gamma = (-\infty, 0]$ in the limit, since the integral on the large circular arc vanishes. See \cite[Section~3.1]{CGMM21} for further discussion on the choice of $\Gamma$.

Regarding the choice between the bound in terms of $\norm{\err_m(w(z))}_2$ and the one in terms of $\norm{\res_m(w(z))}_2$, it is clear that for an a posteriori bound it makes more sense to use the residual norm, since it can be computed quite cheaply once the Krylov basis is available, while the formulation with the error norm requires the exact solution of the linear system, which is significantly more expensive to obtain and therefore impractical.

On the other hand, for an a priori bound neither $\res_m(w(z))$ nor $\err_m(w(z))$ are available because the Krylov basis has not been constructed yet, so in a realistic scenario they have to be bounded using an a priori bound for the linear system error or residual. Which formulation gives the better bound would depend on the quality of the bound for the linear system error or residual, but numerical experiments seem to indicate that the formulation based on the error is usually more accurate. 

The same can be observed for a posteriori bounds as well: the bound in \cref{cor:integral-error-bound--a-posteriori} with the exact error is sharper than the bound with the exact residual. A possible explanation for this fact is that the inequality
\begin{equation*}
	\norm{h_{w, z}(A) \err_m(w)}_2 \le \norm{h_{w, z}}_{[a, b]} \cdot \norm{\err_m(w)}_2
\end{equation*}
is tighter than
\begin{equation*}
	\norm{(A - zI)^{-1} \res_m(w)}_2 \le \norm{k_z}_{[a, b]} \cdot \norm{\res_m(w)}_2,
\end{equation*}
especially when $w$ can be freely chosen (note that the left-hand sides in the two equations above are equal). 
However, this observation is only meaningful for a priori bounds, since in practice one would always use the residual formulation for an a posteriori bound.

A crucial factor for the effective use of the bounds in \cref{cor:integral-error-bound--a-priori,cor:integral-error-bound--a-posteriori} is the choice of the function $w(z)$. The simplest option (and the cheapest for evaluation) is to use a constant function $w(z) \equiv w$, similarly to \cite{CGMM21}, since with this choice we only have to compute or bound the residual (or error) of a single shifted linear system. This turns out to be the most convenient choice for a posteriori bounds, but it often gives underwhelming results when used to obtain a priori bounds, and it may even cause the bounds to diverge (see \cref{subsec:comparison-with-non-optimized-bounds} for an example); in this setting, it is much more effective to use a general function $w(z)$. 

To have a better understanding of this phenomenon, let us focus on the first bound in \cref{cor:integral-error-bound--a-priori}, and let us look at each factor in the integrand separately:

\begin{itemize}
	\item the factor $\abs{ q_{m-1}(z) / q_{m-1}(w) } = \prod_{j = 1}^{m-1} \abs{(z - \xi_j)/(w - \xi_j)}$ is large when $w$ is closer to the set of poles $\{\xi_j\}_j$ compared to $z$, and small when $w$ is farther than $z$ from the poles;
	\item the factor $\norm{h_{w, z}}_{[a, b]}^{m+1}$ is small when $w$ is closer than $z$ to the spectral interval $[a, b]$, and large otherwise;
	\item the factor $\norm{\err_m(w)}_2$ is large when $w$ is close to $[a, b]$, and small when $w$ is close to the poles or when $w$ has a large modulus.
\end{itemize}
A consequence of the different behavior of the factors and their dependence on $z$ is that the value of the parameter $w$ that minimizes the integrand can be significantly different for different values of $z$, so it is difficult to find a single $w$ that makes the integrand small for all $z \in \Gamma$. This is illustrated in \cref{fig:apriori-bound-factors-behavior--2d-plot-1,fig:apriori-bound-factors-behavior--2d-plot-2} with a simple example where the spectrum of the matrix~$A$ is contained in $[0.2, 4]$ and the rational Krylov subspace has four poles, for $z = -4$ and $z = -0.3$ and $w$ that varies in the complex square $[-5, 5] \times [-5i, 5i]$. Note that the poles factors in \cref{fig:apriori-bound-factors-behavior--2d-plot-1,fig:apriori-bound-factors-behavior--2d-plot-2} only differ by a constant multiplicative factor, and that the linear system error factors are exactly the same. However, due to the different behavior of the $h_{w, z}$ factor in the two cases, the full integrands (on the bottom right in \cref{fig:apriori-bound-factors-behavior--2d-plot-1,fig:apriori-bound-factors-behavior--2d-plot-2}) look significantly different.

The values of $\norm{h_{w, z}}_{[a, b]}$ in \cref{fig:apriori-bound-factors-behavior--2d-plot-1,fig:apriori-bound-factors-behavior--2d-plot-2} are computed using \cref{lemma:bound-on-hwz--complex-w-real-z}, which generalizes \cite[Lemma 3.1]{CGMM21} to the case of complex $w$ and real $z$. We mention that this result can also be generalized to the case of both $w$ and $z$ complex with a similar proof, but the statement is less elegant in that case. 
\begin{lemma}
	\label{lemma:bound-on-hwz--complex-w-real-z}
	Let $[a, b] \subset \R$, $z \in \R \setminus [a, b]$ and $w \in \C$. Define
	\begin{equation*}
		\lambda^\star = \frac{\abs{w}^2 - z \Real(w)}{\Real(w) - z}
	\end{equation*}
	and let
	\begin{equation*}
		h^\star = \begin{cases}
			\abs{h_{w, z}(\lambda^\star)} = \abs*{\dfrac{	\Imag(w)}{w-z}} & \text{if $\lambda^\star \in [a, b]$,} \\
			0 & \text{otherwise.}
		\end{cases}
	\end{equation*}
	We have
	\begin{equation*}
		\norm{h_{w, z}}_{[a, b]} = \max \left\{ \abs*{\frac{a-w}{a-z}}, \, \abs*{\frac{b-w}{b-z}}, \, h^\star \right\}.
	\end{equation*}
\end{lemma}
\begin{proof}
	The proof is similar to the proof of \cite[Lemma~3.1]{CGMM21}. For any $\lambda \in \R$, we have
	\begin{equation*}
		H(\lambda) := \abs{h_{w, z}(\lambda)}^2 = \abs*{\frac{\lambda - w}{\lambda - z}}^2 = \frac{(\lambda - \Real(w))^2 + \Imag(w)^2}{(\lambda - z)^2},
	\end{equation*}
	and
	\begin{equation*}
		\frac{\dde H}{\dde \lambda} = \frac{2(\lambda - \Real(w))(\lambda - z)^2 - 2(\lambda - z)( (\lambda - \Real(w))^2 + \Imag(w)^2 )}{(\lambda - z)^4}.
	\end{equation*}
	It follows that $\frac{\dde H}{\dde \lambda} = 0$ only for $\lambda = \lambda^\star$, so the only possible local maxima of $h_{w, z}(\lambda)$ in $[a, b]$ are for $\lambda = a$, $\lambda = b$ and $\lambda = \lambda^\star$, if $\lambda^\star \in [a, b]$. Finally, with some simple algebraic manipulations it can be shown that $\abs{h_{w, z}(\lambda^\star)} = \abs{\Imag(w) / (w-z)}$.
\end{proof}

\begin{figure}[h]
	\centering
	\includegraphics[width=0.46\textwidth]{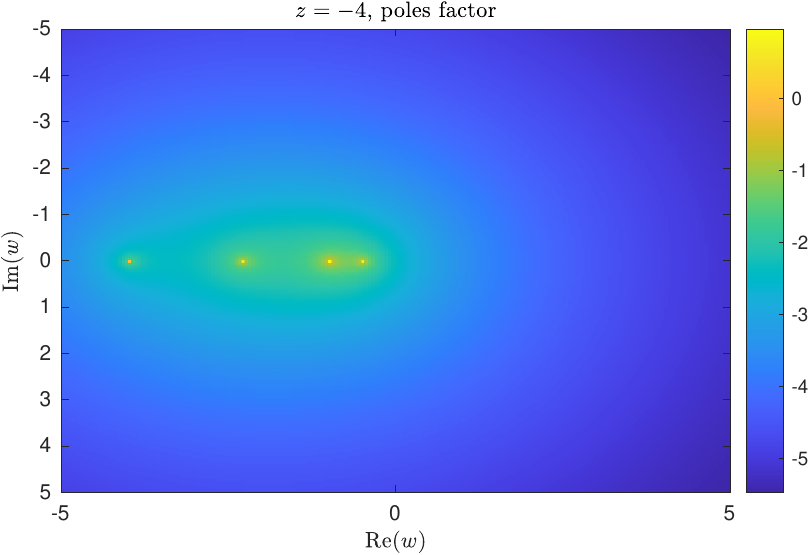}%
	\quad%
	\includegraphics[width=0.46\textwidth]{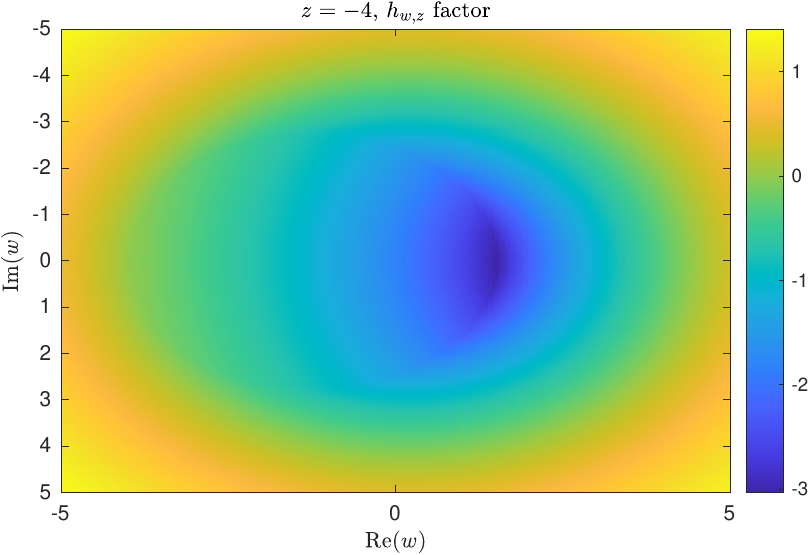}
	\includegraphics[width=0.46\textwidth]{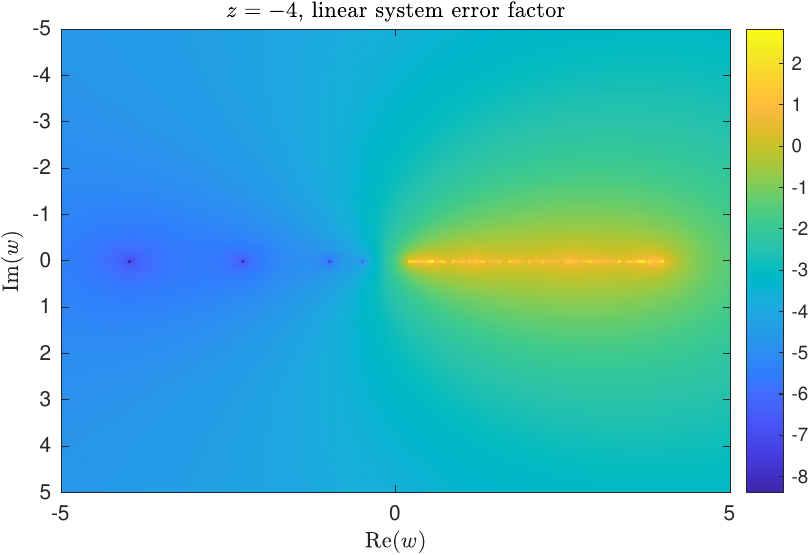}%
	\quad%
	\includegraphics[width=0.46\textwidth]{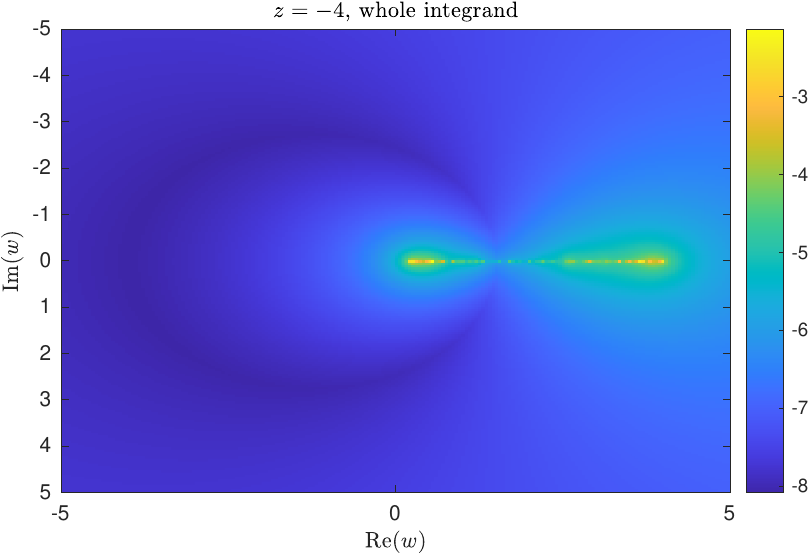}
	\caption{Behavior of the different factors in the first a priori bound in \cref{cor:integral-error-bound--a-priori} for $w \in [-5, 5] \times [-5 i, 5i]$, for $z = -4$.}
	\label{fig:apriori-bound-factors-behavior--2d-plot-1}
\end{figure}

\begin{figure}[h]
	\centering
	\includegraphics[width=0.46\textwidth]{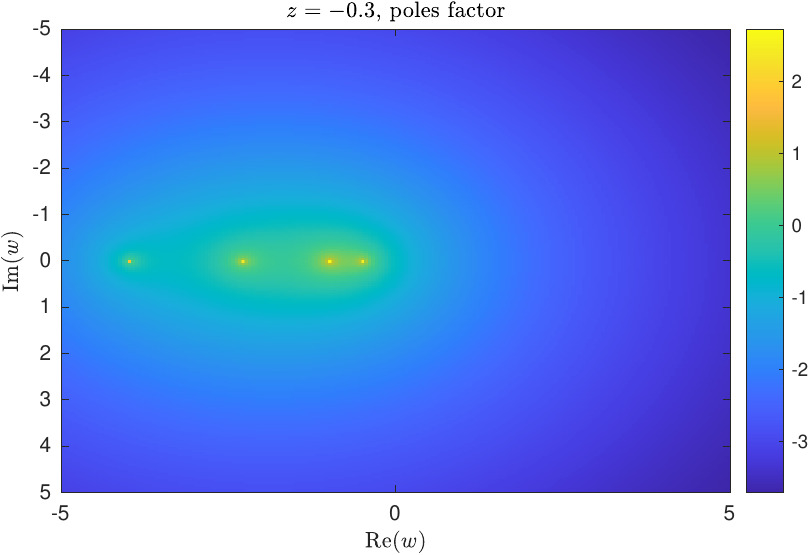}%
	\quad%
	\includegraphics[width=0.46\textwidth]{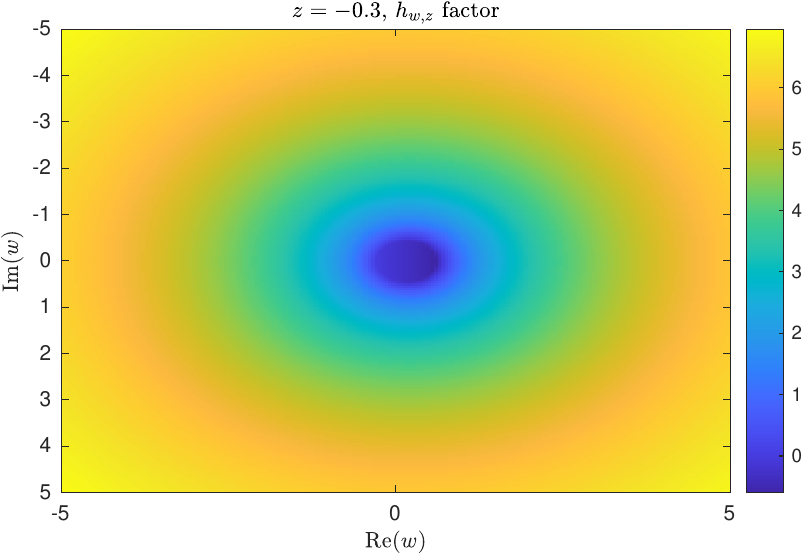}
	\includegraphics[width=0.46\textwidth]{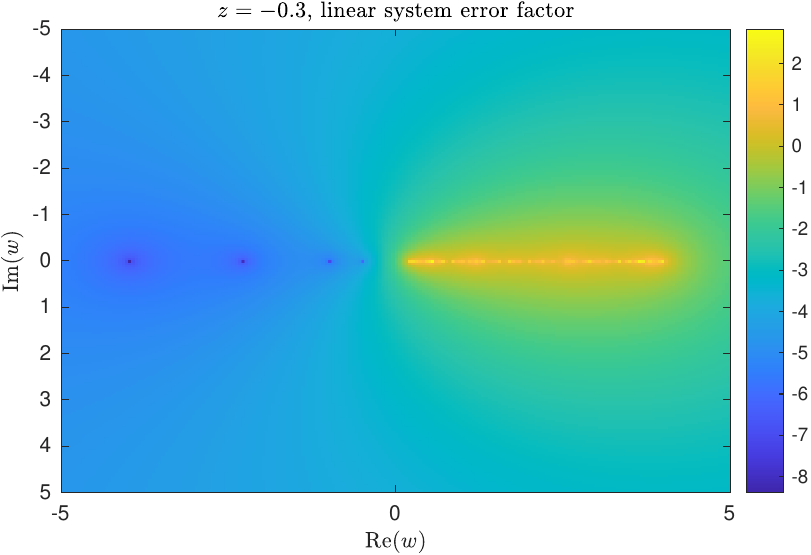}%
	\quad%
	\includegraphics[width=0.46\textwidth]{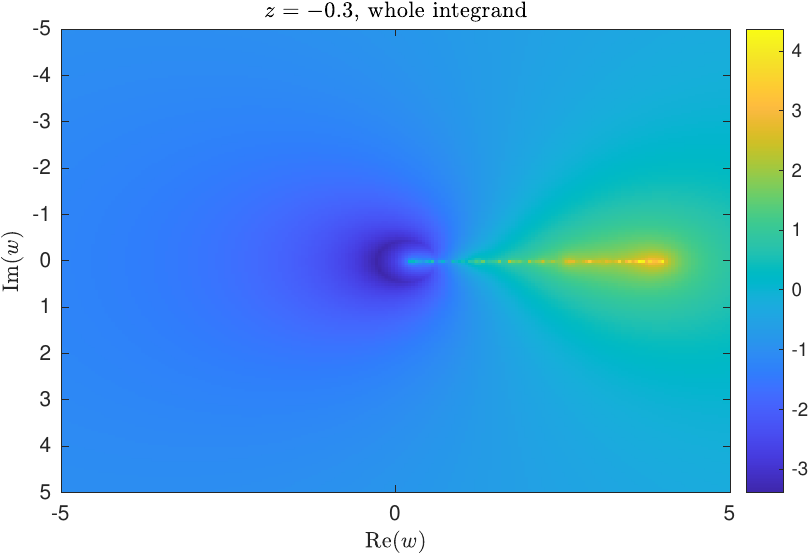}
	\caption{Behavior of the different factors in the first a priori bound in \cref{cor:integral-error-bound--a-priori} for $w \in [-5, 5] \times [-5 i, 5i]$, for $z = -0.3$.}
	\label{fig:apriori-bound-factors-behavior--2d-plot-2}
\end{figure}

We can see from \cref{fig:apriori-bound-factors-behavior--2d-plot-1,fig:apriori-bound-factors-behavior--2d-plot-2} that the values of $w$ for which the integrand is smallest are close to $z$. Indeed, it turns out that the bounds in \cref{cor:integral-error-bound-ratkrylov--general} are minimized precisely when $w(z)\equiv z$.

\begin{proposition}
	\label{prop:bounds-minimized-for-w-equal-to-z}
	Under the same assumptions of \cref{cor:integral-error-bound-ratkrylov--general}, the bounds in \cref{cor:integral-error-bound-ratkrylov--general} are minimized for $w(z) \equiv z$. 
\end{proposition}
\begin{proof}
	Let us consider the first bound in \cref{cor:integral-error-bound-ratkrylov--general}. For $w(z) \equiv z$, we have $h_{z, z}(t) = 1$ for all $t \ne z$ , so the bound becomes simply
	\begin{equation*}
		\norm{f(A) \vec b - \f_m}_2 \le \int_\Gamma \norm{\err_m(z)}_2 \, \abs{\dde \mu(z)}.
	\end{equation*}  
	Now, consider any function $w(z)$ that satisfies the assumptions in \cref{thm:integral-error-ratkrylov--matvec} and recall~\cref{eqn:linsys-error-relation-z-w}, which implies that
	\begin{equation*}
		\norm{\err_m(z)}_2 = \abs*{ \frac{q_{m-1}(z)}{q_{m-1}(w(z))}} \cdot \abs{\det (h_{w(z), z}(A_m))} \cdot \norm{h_{w(z), z}(A) \err_m(w(z))}_2.
	\end{equation*}
	Using the assumptions of \cref{cor:integral-error-bound-ratkrylov--general}, we can bound
	\begin{equation*}
		\abs{\det (h_{w(z), z}(A_m))} \le \prod_{j = 1}^m \norm{h_{w(z), z}}_{\mathcal{S}_j}
	\end{equation*}
	and
	\begin{equation*}
		\norm{h_{w(z), z}(A) \err_m(w(z))}_2 \le \norm{h_{w(z), z}}_{\mathcal{S}_0} \cdot \norm{\err_m(w(z))}_2.
	\end{equation*}
	If we plug these inequalities in the expression for $\norm{\err_m(z)}_2$, we obtain
	\begin{equation*}
		\begin{aligned}
			\norm{f(A) \vec b - \f_m}_2 &\le \int_\Gamma \norm{\err_m(z)}_2 \, \abs{\dde \mu(z)} \\
			&\le \int_\Gamma \abs*{ \frac{q_{m-1}(z)}{q_{m-1}(w(z))}} \cdot \prod_{j = 1}^m \norm{h_{w(z), z}}_{\mathcal{S}_j} \cdot \norm{h_{w(z), z}}_{\mathcal{S}_0} \cdot \norm{\err_m(w(z))}_2 \, \abs{\dde \mu(z)},
		\end{aligned}	
	\end{equation*}
	which coincides with the first bound in \cref{cor:integral-error-bound-ratkrylov--general}. Since this holds for any admissible function $w(z)$, the function $w(z) = z$ must be a minimizer for the right-hand side of the bound. 

	The same can be proved for the second bound in \cref{cor:integral-error-bound-ratkrylov--general}, by observing that for~$w(z) = z$ the bound simplifies to
	\begin{equation*}
		\norm{f(A) \vec b - \f_m}_2 \le \int_\Gamma \norm{k_z}_{\mathcal{S}_0} \cdot \norm{\res(z)}_2 \, \abs{\dde \mu(z)}
	\end{equation*} 
	and then using the same strategy, recalling the identity
	\begin{equation*}
		\norm{\res(z)}_2 = \abs*{ \frac{q_{m-1}(z)}{q_{m-1}(w(z))}} \cdot \abs{\det (h_{w(z), z}(A_m))} \cdot \norm{\res(w(z))}_2.
	\end{equation*}%
\end{proof}

Although using $w(z)\equiv z$ is not very practical because it would require the computation of the residual or error norm of a different shifted linear system for each point used in the discretization of the integral, the result of \cref{prop:bounds-minimized-for-w-equal-to-z} provides some theoretical insight regarding the choice of the function $w(z)$. 
In the following subsections we discuss practical ways to choose the function $w(z)$ in order to approximately minimize the bounds, both in the a priori and a posteriori setting.

\subsection{A priori bounds}
\label{subsec:a-priori-bounds}

Let us consider the first a priori bound in \cref{cor:integral-error-bound--a-priori}. For any~$z \in \Gamma$, the ideal choice for $w(z)$ would be one that minimizes the integrand in the bound, i.e. 
\begin{equation*}
	w(z) = \argmin_{w \in \C \setminus \mathcal{D}} \left( \frac{1}{\abs{q_{m-1}(w)}} \cdot \norm{h_{w, z}}_{[a, b]}^{m+1} \cdot \norm{\err_m(w)}_2 \right),
\end{equation*}
where we have dropped the factors that do not depend on $w$, and the set $\mathcal{D}$ is the one defined in \cref{thm:integral-error-ratkrylov--matvec}.
We have seen in \cref{prop:bounds-minimized-for-w-equal-to-z} that the bound is minimized by taking $w(z) = z$. However, in practice we do not have access to $\norm{\err_m(w)}_2$ a priori, so we must instead rely on an a priori bound for the linear system error of the form $\norm{\err_m(w)}_2 \le \varphi_m(w)$, and then select
\begin{equation}
	\label{eqn:choice-of-w--a-priori}
	w(z) = \argmin_{w \in \C \setminus \mathcal{D}} \left( \frac{1}{\abs{q_{m-1}(w)}} \cdot \norm{h_{w, z}}_{[a, b]}^{m+1} \cdot \varphi_m(w) \right).
\end{equation} 
In general, the choice $w(z) = z$ is not going to be the minimizer of \cref{eqn:choice-of-w--a-priori}, but we may heuristically expect that choosing $w(z)$ near $z$ gets us close to the optimum. The optimization problem can be approximately solved numerically by replacing $\C$ with a finite set $\mathcal{W}$ of candidates for $w$, and by finding for each $z$ the value of $w \in \mathcal{W}$ that minimizes the function on the right-hand side. 

Note that in a practical implementation of the bounds, the integral will be approximately evaluated using a quadrature formula, so the number of values of $z \in \Gamma$ for which we have to find the best $w \in \mathcal{W}$ will be finite.
For each quadrature node $z$, the minimization over the set~$\mathcal{W}$ requires the computation of $\norm{h_{w, z}}_{[a, b]}$ for all $w \in \mathcal{W}$, which can be done efficiently using \cite[Lemma~3.1]{CGMM21} if $w$ is real, or with \cref{lemma:bound-on-hwz--complex-w-real-z} if $z$ is real. Additionally, for each $w \in \mathcal{W}$ we have to compute $q_{m-1}(w)$ and $\varphi_m(w)$. 
Although these computations may significantly increase the cost of quadrature if the cardinality of $\mathcal{W}$ is large, it is important to note that their cost is independent of the matrix dimension $n$, and therefore it eventually becomes negligible compared to the cost of constructing the Krylov basis as the size of $A$ increases. 

A possible choice for the set $\mathcal{W}$ is a discretization of the integration contour $\Gamma$, which ensures that for each $z \in \Gamma$ there exists a $w \in \mathcal{W}$ that is quite close. Alternatively, one could decide to use a different set $\mathcal{W}_z$ for each $z \in \Gamma$, for instance by taking a small number of points in the neighborhood of $z$, or even a single point, such as the $w \in \mathcal{W}$ that is closest to $z$. However, this kind of choice heavily relies on the heuristic that the optimum of \cref{eqn:choice-of-w--a-priori} is attained for $w$ close to $z$, which is not always true. For example, this heuristic may fail when the upper bound $\norm{\err_m(w)}_2 \le \varphi_m(w)$ significantly overestimates the error norm only for certain values of~$w$.

\subsection{A posteriori bounds}
\label{subsec:a-posteriori-bounds}
Let us consider now the a posteriori bounds. We focus on the second bound in \cref{cor:integral-error-bound--a-posteriori}, since the formulation with the linear system residual is the most likely to be useful in practice. In this case, the best choice for the function $w(z)$ would be one that minimizes the integrand in the bound, i.e.
\begin{equation*}
	w(z) = \argmin_{w \in \C \setminus \mathcal{D}} \frac{\abs{\chi_m(w)}}{\abs{q_{m-1}(w)}} \cdot \norm{\res_m(w)}_2,
\end{equation*} 
where again we have only included the terms that depend on $w$. 
However, it follows from~\cref{eqn:linsys-residual-relation-z-w} that the quantity $\res_m(w) \cdot \chi_m(w) / q_{m-1}(w)$ is constant in $w$, so any value of~$w$ can be used to evaluate the bound and would give the same result. The most efficient choice is then to use a single $w$ for all $z \in \Gamma$. Note that with the choice $w(z) = z$, the bound can be rewritten in the simple form 
\begin{equation*}
	\norm{f(A) \vec b - \f_m} \le \int_\Gamma \norm{(A - zI)^{-1}}_2 \cdot \norm{\res_m(z)}_2 \, \abs{\de \mu(z)},
\end{equation*}
which can be easily obtained directly from \cref{eqn:basic-integral-error-ratkrylov--matvec}. The advantage of the formulation with $w(z) \equiv w$ is that the corresponding bound can be evaluated by computing the residual of a single linear system.

The evaluation of the a posteriori bound after $m$ Krylov iterations requires the computation of the residual
\begin{equation*}
	\res_m(w) = (A - wI) \sol_m(w) = (A - wI)V_m (A_m - wI)^{-1} \vec e_1 \norm{\vec b}_2.
\end{equation*}
Note that it is likely that an eigendecomposition of $A_m$ has been already computed for the evaluation of $f(A_m)$, so the $m \times m$ linear system $(A_m - wI)^{-1} \vec e_1$ can be solved with just $O(m^2)$ cost, with no need to compute a factorization of the matrix $A_m - w I$. So we can compute $\sol_m(w)$ for $O(mn)$ cost, and we can obtain $\res_m(w)$ with one additional matrix-vector multiplication with $A$. 
Alternatively, if we can keep in memory the matrix $A V_m$ (whose columns are often computed anyway in the rational Arnoldi algorithm), we can entirely avoid additional matrix-vector multiplications with $A$ and cheaply compute $\res_m(w)$ for a cost of $O(mn)$.  

\section{An a priori bound for the linear system residual}
\label{sec:linear-system-bounds}

In this section we present an a priori bound on the residual of a shifted linear system, which can be used to bound $\norm{\res_m(w(z))}_2$ and $\norm{\err_m(w(z))}_2$ in~\cref{cor:integral-error-bound--a-priori}. 

Consider the shifted linear system
\begin{equation*}
	(A - w I) \vec x = \vec b, \qquad w \in \C.
\end{equation*}
This linear system coincides with a Sylvester equation in which one of the coefficient matrices is $1 \times 1$, 
\begin{equation*}
	A X - X B = \vec b \vec c^T,
\end{equation*}
where $X = \vec x \in \R^{n \times 1}$, $B = w \in \R^{1 \times 1}$ and $\vec c = 1 \in \R$. 
As a consequence, the residual norm of the shifted linear system can be bounded by using the bounds developed in \cite{Beckermann11} for the residual of Sylvester equations.
The theorem that follows is obtained by specializing~\cite[Theorem~2.3]{Beckermann11} to the case of a shifted linear system with a symmetric positive definite matrix~$A$. 
The notation in the statement of the theorem is changed from \cite{Beckermann11} in order to match the notation used in the rest of the paper.

\begin{theorem} \cite[Theorem~2.3]{Beckermann11}
	\label{thm:apriori-linsys-bound--residual-sylvester}
	Assume that $A$ is symmetric with spectrum contained in $[\lambda_\text{min}, \lambda_\text{max}]$, and let the set of poles $\{ \xi_1, \dots, \xi_{m-1} \}$ be closed under complex conjugation. For any $w \in \R \setminus[\lambda_{\text{min}}, \lambda_{\text{max}}] \cup \{\xi_1, \dots, \xi_{m-1}\}$, we have
	\begin{equation*}
		\norm{\res_m(w)}_2 \le \norm{\vec b}_2 \, (4 + c) \, \gamma_m,
	\end{equation*}
	where
	\begin{equation*}
		c = 2 \sqrt{2} \sqrt{\kappa_2(A - wI)}
	\end{equation*}
	and
	\begin{equation*}
		\gamma_m = \left|\frac{\varphi(w) - 1}{\varphi(w) + 1} \right| \cdot \prod_{j = 1}^{m-1} \left|\frac{\varphi(w)/\varphi(\xi_j) - 1}{\varphi(w)/\varphi(\conj{\xi_j}) + 1} \right|, \qquad \varphi(z) = \sqrt {\dfrac{ z - \lambda_{\text{max}}}{z - \lambda_{\text{min}}}} \, .
	\end{equation*}
\end{theorem}
\begin{proof}
	We show that this theorem follows by applying the second part of \cite[Theorem~2.3]{Beckermann11} to the linear system $(A - w I) \vec x = \vec b$, interpreted as a Sylvester equation. In this proof we use some notation from \cite{Beckermann11}.

	Note that in our setting we have $B = w \in \R$, so the size of the Krylov subspace associated to $B$ is $n = 1$, and the corresponding pole is $z_{B, 1} = \infty$. On the other hand, the poles $\{ z_{A, 1}, \dots, z_{A, m} \}$ used for the rational Krylov subspace associated to $A$ are given by $\{\infty, \xi_1, \dots, \xi_{m-1}\}$, where the pole at infinity appears because of the different notation used in the definition of rational Krylov subspaces in \cite{Beckermann11}.

	The numerical range of $B$ is $W(B) = \{ w \}$, so we have $u_{B, n} \equiv 0$ and
	\begin{equation*}
		\gamma_{B, A} = \max_{z \in [\lambda_\text{min}, \lambda_\text{max}]} u_{B, n}(z) = 0.
	\end{equation*} 
	We have 
	\begin{equation*}
		c_3 = 2\sqrt{2} \sqrt{\frac{\max\{\abs{\lambda_{\text{min}} - w}, \abs{\lambda_{\text{max}} - w}\}}{\min\{ \abs{\lambda_{\text{min}} - w}, \abs{\lambda_{\text{max}} - w} \}}} = 2 \sqrt{2} \sqrt{\kappa_2(A - wI)} =: c,
	\end{equation*}
	and 
	\begin{equation*}
		\gamma_{A, B} = u_{A, m}(w) = \left|\frac{\varphi(w) - 1}{\varphi(w) + 1} \right| \cdot \prod_{j = 1}^{m-1} \left|\frac{\varphi(w)/\varphi(\xi_j) - 1}{\varphi(w)/\varphi(\conj{\xi_j}) + 1} \right| =: \gamma_m, \qquad \varphi(z) := \sqrt {\dfrac{ z - \lambda_{\text{max}}}{z - \lambda_{\text{min}}}} \, ,
	\end{equation*}
	where the first factor in $\gamma_m$ corresponds to the pole at infinity.

	Observing that the Sylvester equation residual $\norm{S_{A, B}(X - X^G_{m, n})}_F$ from \cite{Beckermann11} corresponds to the shifted linear system residual $\norm{\res_m(w)}_2$ in our setting, by \cite[Theorem~2.3]{Beckermann11} we can conclude that
	\begin{align*}
		\norm{\res_m(w)}_2 &\le \norm{\vec b}_2 \, (4 \, \max\{ \gamma_{A, B} + \gamma_{B, A} \} + c_3 \, ( \gamma_{A, B} + \gamma_{B, A} )) \\
		&= \norm{\vec b}_2 \, (4 + c_3) \, \gamma_m.
	\end{align*} 
\end{proof}

\begin{remark}
	\label{rem:linsys-bound--non-positive-definite-case}
	The original statement of \cite[Theorem~2.3]{Beckermann11} holds also for nonsymmetric matrices, by using the numerical ranges of $A$ and $B$ and the associated Green functions. The main assumption on the matrices $A$ and $B$ is that $W(A) \cap W(B) = \emptyset$, which in our setting corresponds to $w \notin W(A)$. Here we prefer to consider only the symmetric case, since it is difficult to compute the numerical range of a general matrix in practice.
\end{remark}
\begin{remark}
	\label{rem:linsys-bound-from-residual-to-error}
	The bound on the residual norm $\norm{\res_m(w)}_2$ also provides a bound for the error norm $\norm{\err_m(w)}_2$ via the inequality
	\begin{equation}
		\label{eqn:linsys-bound--residual-to-error}
		\begin{aligned}
			\norm{\err_m(w)}_2 &\le \norm{(A - w I)^{-1}}_2 \cdot \norm{\res_m(w)}_2 \\
			&\le \max\{\abs{\lambda_{\text{min}} - w}^{-1}, \abs{\lambda_{\text{max}} - w}^{-1} \} \cdot \norm{\res_m(w)}_2.
		\end{aligned}
	\end{equation}
	The right-hand side of the inequality \cref{eqn:linsys-bound--residual-to-error} can be used in \cref{eqn:choice-of-w--a-priori} as the function $\varphi_m(w)$.
\end{remark}

\section{Error bounds for quadratic forms}
\label{sec:quadratic-forms}

In this section we adapt the statement of \cref{thm:integral-error-ratkrylov--matvec} to the case of quadratic forms with $f(A)$, and we show that it leads to improved bounds when $A$ is symmetric. 
Let us denote by $\qf_m = \vec b^T \f_m$ the approximation to $\vec b^T f(A) \vec b$ obtained after $m$ steps of a rational Krylov method. From \cref{eqn:basic-integral-error-ratkrylov--matvec} we immediately have the identity
\begin{equation*}
	\vec b^T f(A) \vec b - \qf_m = \int_\Gamma \vec b^T \err_m(z) \de \mu(z).
\end{equation*}
We can write $\vec b = (A - z I) \sol_m(z) + \res_m(z)$, and using the symmetry of $A$ we have
\begin{align*}
	\vec b^T \err_m(z) &= \sol_m(z)^T (A - z I) \err_m(z) + \res_m(z)^T \err_m(z) \\
	&= \sol_m(z)^T \res_m(z) + \res_m(z)^T \err_m(z) \\
	&= \res_m(z)^T \err_m(z),
\end{align*}
where in the last equality we used the fact that $\res_m(z) \perp \rat_m(A, \vec b)$, so $\sol_m(z)^T\res_m(z) = 0$; see \cref{eqn:shifted-residuals-collinearity}.
The error for the approximation of $\vec b^T f(A) \vec b$ is therefore given by
\begin{equation}
	\label{eqn:basic-integral-error-ratkrylov--quadform}
	\vec b^T f(A) \vec b - \qf_m = \int_\Gamma \res_m(z)^T \err_m(z) \de \mu(z).
\end{equation}
Combining this error expression with \cref{eqn:linsys-residual-relation-z-w} leads to the following theorem, which is the equivalent of \cref{thm:integral-error-ratkrylov--matvec} for quadratic forms.

\begin{theorem}
	\label{thm:integral-error-ratkrylov--quadform}
	Assume that $A$ is symmetric and that $f$ has the integral representation \cref{eqn:integral-representation-f}, and denote by $\qf_m$ the approximation to $\vec b^T f(A) \vec b$ after $m$ iterations of a rational Krylov method with denominator polynomial $q_{m-1}$. Let $\mathcal{D} = \{\xi_1, \dots, \xi_{m-1} \} \cup \sigma(A) \cup \sigma(A_m)$ and let $w$ be a function $w : \Gamma \to \C \setminus \mathcal{D}$. We have
	\begin{equation*}
		\begin{aligned}
			\vec b^T f(A) \vec b - \qf_m &= \int_\Gamma \bigg( \frac{q_{m-1}(z)}{q_{m-1}(w(z))} \det(h_{w(z), z}(A_m)) \bigg)^2 \res_m(w(z))^T k_z(A) \res_m(w(z)) \de \mu(z) \\
			&= \int_\Gamma \bigg( \frac{q_{m-1}(z)}{q_{m-1}(w(z))} \det(h_{w(z), z}(A_m)) \bigg)^2 \res_m(w(z))^T h_{w(z), z}(A) \err_m(w(z)) \de \mu(z).
		\end{aligned}
	\end{equation*}
\end{theorem}
\begin{proof}
	In \cref{eqn:basic-integral-error-ratkrylov--quadform}, use $\err_m(z) = (A - z I)^{-1} \res_m(z)$ and use \cref{eqn:linsys-residual-relation-z-w} to replace $\res_m(z)$ with $\res_m(w)$.
\end{proof}

The following corollary can be derived from the two error expressions in \cref{thm:integral-error-ratkrylov--quadform} with a proof similar to the one of \cref{cor:integral-error-bound-ratkrylov--general}.

\begin{corollary}
	\label{cor:integral-error-bound-ratkrylov--general-quadform}
	With the same hypothesis as in \cref{thm:integral-error-ratkrylov--quadform}, assume that for some $\mathcal{S}_0, \,\mathcal{S}_1, \dots, \mathcal{S}_m \subset \C$ we have $\sigma(A) \subset \mathcal{S}_0$, and $\lambda_i(A_m) \in \mathcal{S}_i$ for all $i = 1, \dots, m$. Then we have
	\begin{equation*}
		\norm{\vec b^T f(A) \vec b - \qf_m}_2 \le \int_\Gamma \abs*{ \frac{q_{m-1}(z)}{q_{m-1}(w(z))}}^2 \prod_{j = 1}^m \norm{h_{w(z), z}}^2_{\mathcal{S}_j} \cdot \norm{k_z}_{\mathcal{S}_0} \cdot \norm{\res_m(w(z))}^2_2 \, \abs{\dde \mu(z)}
	\end{equation*}
	and similarly
	\begin{equation*}
		\norm{\vec b^T f(A) \vec b - \qf_m}_2 \le \int_\Gamma G_m(z) \, \abs{\dde \mu(z)},
	\end{equation*}
	where 
	\begin{equation*}
		G_m(z) = \abs*{ \frac{q_{m-1}(z)}{q_{m-1}(w(z))}}^2 \prod_{j = 1}^m \norm{h_{w(z), z}}^2_{\mathcal{S}_j} \cdot \norm{h_{w(z), z}}_{\mathcal{S}_0} \cdot \norm{\res_m(w(z))}_2 \cdot \norm{\err_m(w(z))}_2.
	\end{equation*}
\end{corollary}

Note that the first bound in \cref{cor:integral-error-bound-ratkrylov--general-quadform} has the same structure as the second one in \cref{cor:integral-error-bound-ratkrylov--general}, with the difference that several of the factors are squared. This is in line with the fact that the error $\norm{\vec b^T f(A) \vec b - \qf_m}_2$ is related to the error in the approximation of $f$ with rational functions of type $(2m-1, 2m-2)$, instead of type $(m-1, m-1)$, so we can expect to have roughly twice the convergence rate compared to the error $\norm{f(A) \vec b - \f_m}_2$ (see e.g. \cite[Remark~3.2]{Guettel13} or \cite[Proposition~4.4]{BRS22}). A similar result has been obtained in \cite[Section~6]{CGMM21} for the Lanczos method.

\cref{cor:integral-error-bound-ratkrylov--general-quadform} can be specialized by choosing different sets $\mathcal{S}_j$ to obtain a priori and a posteriori bounds. This is formalized in the following two corollaries.
\begin{corollary}
	\label{cor:integral-error-bound--a-priori-quadform}
	Under the assumptions of \cref{thm:integral-error-ratkrylov--quadform}, we have the a priori bounds
	\begin{equation*}
		\norm{\vec b^T f(A) \vec b - \qf_m}_2 \le \int_\Gamma \abs*{ \frac{q_{m-1}(z)}{q_{m-1}(w(z))}}^2 \norm{h_{w(z), z}}_{[a, b]}^{2m} \cdot \norm{k_z}_{[a, b]} \cdot \norm{\res_m(w(z))}_2^2 \, \abs{\dde \mu(z)}
	\end{equation*}
	and
	\begin{equation*}
		\norm{\vec b^T f(A) \vec b - \qf_m}_2 \le \int_\Gamma H_m(z) \,\abs{\dde \mu(z)},
	\end{equation*}
	where
	\begin{equation*}
		H_m(z) = \abs*{ \frac{q_{m-1}(z)}{q_{m-1}(w(z))}}^2 \prod_{j = 1}^m \norm{h_{w(z), z}}^{2m+1} \cdot \norm{\res_m(w(z))}_2 \cdot \norm{\err_m(w(z))}_2.
	\end{equation*}
\end{corollary}

\begin{corollary}
	\label{cor:integral-error-bound--a-posteriori-quadform}
	Under the assumptions of \cref{thm:integral-error-ratkrylov--quadform}, we have the a posteriori bound
	\begin{equation*}
		\norm{\vec b^T f(A) \vec b - \qf_m}_2 \le \int_\Gamma \abs*{ \frac{q_{m-1}(z)}{q_{m-1}(w(z))} \cdot \frac{\chi_m(w(z))}{\chi_m(z)}}^2 \cdot \norm{k_z}_{[a, b]} \cdot \norm{\res_m(w(z))}_2^2 \, \abs{\dde \mu(z)}. 
	\end{equation*}
\end{corollary}

\begin{remark}
	\label{rem:quadform-bounds-discussion}
	With the same arguments used in the proof of \cref{prop:bounds-minimized-for-w-equal-to-z}, it can be shown that the bounds in \cref{cor:integral-error-bound-ratkrylov--general-quadform} are minimized by taking $w(z) = z$, and with this choice the bounds have much simpler expressions that can be derived directly from~\cref{eqn:basic-integral-error-ratkrylov--quadform}. From these expressions, it is easy to see that for $w(z) = z$ the second bound in \cref{cor:integral-error-bound--a-priori-quadform} is always sharper than the first one, because of the inequality $\norm{\err_m(z)}_2 \le \norm{k_z}_{[a, b]} \norm{\res_m(z)}_2$. 
	In the experiments of \cref{sec:numerical-experiments} we observe this also when $w(z)$ is an approximate minimizer of the integrand, assuming that the exact linear system error and residual norms are known. On the other hand, when using a priori bounds on the linear system error and residual norms, which bound is better depends on the linear system bounds that we use.
	In particular, if we have an a priori bound on the linear system residual and we use it obtain a bound on the error via \cref{eqn:linsys-bound--residual-to-error}, then the first bound in \cref{cor:integral-error-bound--a-priori-quadform} is always sharper than the second one, because of the inequality $\norm{k_z}_{[a, b]} \le \norm{h_{w, z}}_{[a, b]} \norm{k_w}_{[a, b]}$.
\end{remark}

Let us briefly discuss how to select the function $w(z)$. 
In the case of a priori bounds, similarly to the case of $f(A) \vec b$, we should select $w(z)$ as a minimizer of the integrand evaluated at $z$, where in practice $\norm{\res_m(w)}_2$ is replaced by an a priori bound of the form $\psi_m(w) \ge \norm{\res_m(w)}_2$, and $\norm{\err_m(w)}_2$ is replaced by $\varphi_m(w) \ge \norm{\err_m(w)}_2$. In the end, we obtain
\begin{equation*}
	w(z) = \argmin_{w \in \C \setminus \mathcal{D}} \left( \frac{1}{\abs{q_{m-1}(w)}^2} \cdot \norm{h_{w, z}}_{[a, b]}^{2m} \cdot \psi_m(w)^2 \right).
\end{equation*}
for the first bound in \cref{cor:integral-error-bound--a-priori-quadform}, and
\begin{equation*}
	w(z) = \argmin_{w \in \C \setminus \mathcal{D}} \left( \frac{1}{\abs{q_{m-1}(w)}^2} \cdot \norm{h_{w, z}}_{[a, b]}^{2m+1} \cdot \varphi_m(w) \psi_m(w) \right).
\end{equation*}
for the second bound.
In practice, the minimization problem can be solved approximately by taking $w$ in a finite set $\mathcal{W}$, which can be chosen for example as a discretization of $\Gamma$, similarly to the case of bounds for $f(A) \vec b$. 

When considering a posteriori bounds the situation is simpler, since by \cref{eqn:linsys-residual-relation-z-w} for any $w \in \C \setminus \mathcal{D}$ we have 
\begin{equation*}
	\left( \frac{q_{m-1}(z)}{q_{m-1}(w)} \cdot \frac{\chi_m(w)}{\chi_m(z)} \right) \res_m(w) = \res_m(z),
\end{equation*}
so we can choose $w(z) \equiv w$ for any fixed $w$ and obtain the same bound.

\section{Numerical experiments}
\label{sec:numerical-experiments}

In this section we include some experiments that demonstrate the performance of the error bounds for the approximation of $f(A) \vec b$ and $\vec b^T f(A) \vec b$ when the function $f$ is either Cauchy-Stieltjes or analytic in a neighborhood of the spectrum of $A$. 
We use the functions $f(z) = z^{-1/2}$ and $f(z) = \log(1+z)/z$ with the Cauchy-Stieltjes formulations in \cref{eqn:cauchy-stieltjes-examples}, and the function $f(z)= z^{1/2}$ with the Cauchy integral formula on the contour~$\Gamma = (-\infty, 0]$. 
In all experiments, the poles used to construct the rational Krylov subspace are the asymptotically optimal poles for Cauchy-Stieltjes functions proposed in \cite{MasseiRobol21}.

\subsection{Comparison with non-optimized bounds} 
\label{subsec:comparison-with-non-optimized-bounds}

We first conduct a simple experiment to show that the optimization of the parameter $w$ discussed in \cref{subsec:a-priori-bounds} is fundamental for the applicability of a priori bounds. We consider the computation of $A^{-1/2} \vec b$, for a $1000 \times 1000$ symmetric matrix $A$ with logspaced eigenvalues in the interval $[10^{-1}, 10^1]$, and we compare the error with the residual-based a posteriori bound from \cref{cor:integral-error-bound--a-posteriori}, and the error-based a priori bounds from \cref{cor:integral-error-bound--a-priori} for some fixed values of $w$ and with the optimization described in \cref{subsec:a-priori-bounds}. The error norm in the a priori bounds is bounded using \cref{eqn:linsys-bound--residual-to-error} and \cref{thm:apriori-linsys-bound--residual-sylvester}. The results are depicted in \cref{fig:a-priori-optimized-vs-not-opt}. It turns out that the a priori bounds without optimization diverge in this case, and they give no useful information. On the other hand, the bound obtained by optimizing $w$ for each $z \in \Gamma$ used to evaluate the integral correctly captures the convergence behavior, although it is off by a couple of order of magnitude. The a posteriori bound requires no optimization, confirming the discussion in~\cref{subsec:a-posteriori-bounds}.

\begin{figure}[h]
	\centering
	\includegraphics[width=0.65\textwidth]{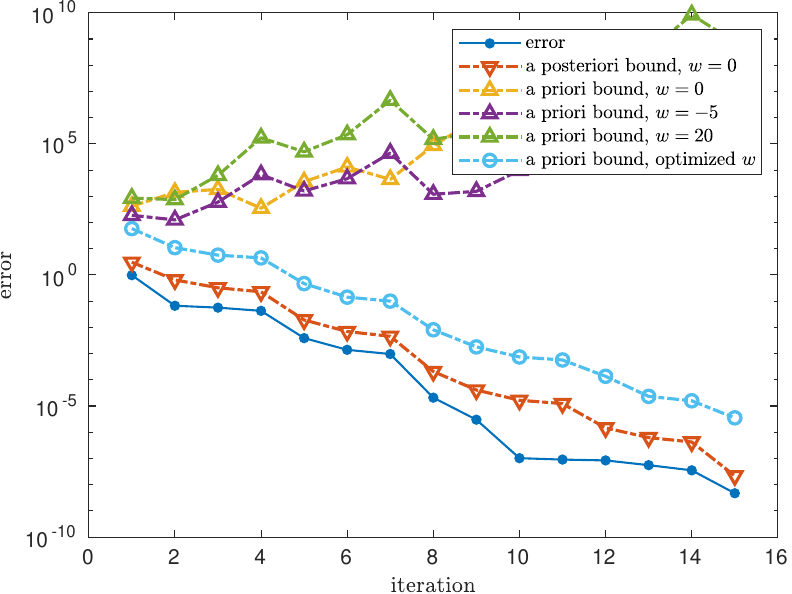}%
	\caption{A priori and a posteriori error bounds for $A^{-1/2} \vec b$, where $A$ is a symmetric $1000 \times 1000$ matrix with logspaced eigenvalues in $[10^{-1}, 10^1]$.}
	\label{fig:a-priori-optimized-vs-not-opt}
\end{figure}

\subsection{Bounds for matrix-vector products} 

We compare the following bounds for the error in the computation of $f(A) \vec b$:
\begin{itemize}
	\item the residual-based a posteriori bound from \cref{cor:integral-error-bound--a-posteriori} with $w = 0$,
	\item the error-based a priori bound from \cref{cor:integral-error-bound--a-priori} with optimized $w$ and linear system error bounded with \cref{eqn:linsys-bound--residual-to-error} and \cref{thm:apriori-linsys-bound--residual-sylvester},
	\item the error-based a posteriori bound from \cref{cor:integral-error-bound--a-posteriori} with optimized $w$ and exact linear system error,
	\item the error-based a priori bound with optimized $w$ and exact linear system error.
\end{itemize}
Note that the third and especially the fourth bound cannot be used in a practical setting, but they can provide interesting theoretical insight. 
Recall that the residual-based a posteriori bound is independent of the choice of $w$, so we can simply take $w = 0$ (see \cref{subsec:a-posteriori-bounds}), but the error-based a posteriori bound does not have this property and hence benefits from the optimization of~$w$, similarly to the a priori bounds. 

The results for the two functions $f(z) = \log(1+z)/z$ and $f(z) = \sqrt{z}$ are shown in \cref{fig:compare-bounds--matvec}. 
We notice that the a priori and a posteriori bounds that use the exact error are slightly better than the residual-based a posteriori bound, and they almost overlap with each other. 
By comparing the a priori bounds that use the exact linear system error with the ones that use the bound in \cref{thm:apriori-linsys-bound--residual-sylvester}, we can conclude that most of the overestimation in the a priori error bounds for $f(A) \vec b$ comes from the overestimation of the linear system error when combining \cref{eqn:linsys-bound--residual-to-error} with \cref{thm:apriori-linsys-bound--residual-sylvester}. Indeed, if we had access to the exact linear system errors in advance, we would be able to obtain a priori the fourth bound in \cref{fig:compare-bounds--matvec}, which captures the convergence behavior very well. Although this scenario is clearly not realistic, an interesting consequence of this observation is that any improvement to a priori bounds for the linear system error would automatically lead to an equal improvement in the a priori bounds for $f(A) \vec b$ of \cref{cor:integral-error-bound--a-priori}, with the limit case of an exact linear system error yielding a bound for $f(A) \vec b$ that is quite close to the true error.

\begin{figure}[h]
	\centering
	\includegraphics[width=0.49\textwidth]{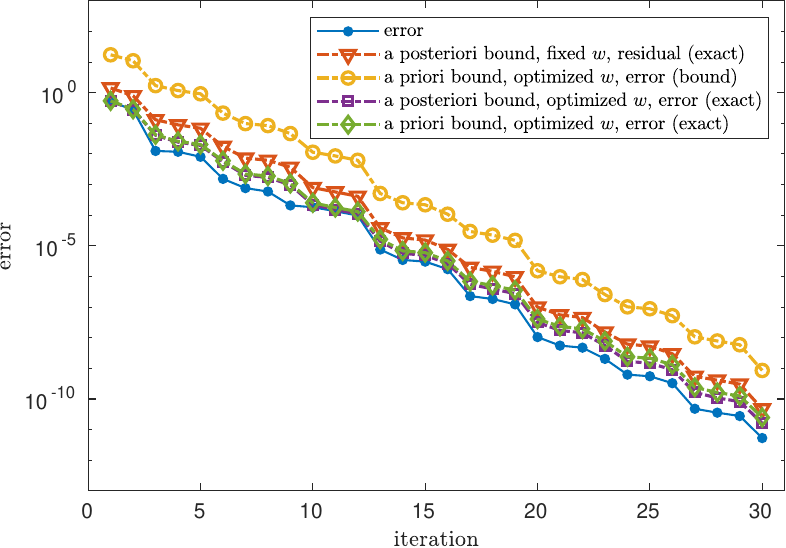}%
	\:%
	\includegraphics[width=0.49\textwidth]{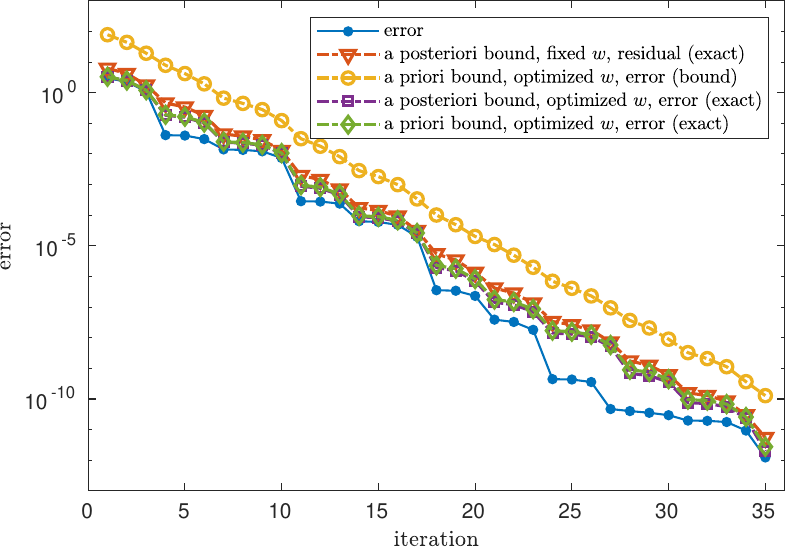}
	\caption{A priori and a posteriori bounds for $f(A) \vec b$, where $A$ is a $1000 \times 1000$ symmetric matrix with logspaced eigenvalues in $[10^{-2}, 10^2]$. Left: $f(z) = \log(1+z)/z$. Right: $f(z) = \sqrt{z}$.}
	\label{fig:compare-bounds--matvec}
\end{figure}

\subsection{Bounds for quadratic forms}
We compare the following bounds for the error in the computation of the quadratic form $\vec b^T f(A) \vec b$:
\begin{itemize}
	\item the a posteriori bound from \cref{cor:integral-error-bound--a-posteriori-quadform} with $w = 0$, using the exact linear system residual,
	\item the first a priori bound from \cref{cor:integral-error-bound--a-priori-quadform} with optimized $w$ and linear system residual bounded using \cref{thm:apriori-linsys-bound--residual-sylvester},
	\item the first a priori bound from \cref{cor:integral-error-bound--a-priori-quadform} with optimized $w$ and exact linear system residual,
	\item the second a priori bound from \cref{cor:integral-error-bound--a-priori-quadform} with optimized $w$ and exact linear system error and residual.
\end{itemize} 
Recall that the second bound in \cref{cor:integral-error-bound--a-priori-quadform} with $\norm{\res_m(w)}_2$  bounded with \cref{thm:apriori-linsys-bound--residual-sylvester} and $\norm{\err_m(w)}_2$ bounded using \cref{eqn:linsys-bound--residual-to-error} is always less accurate than the first one (see \cref{rem:quadform-bounds-discussion}), so we omit it from our experiments. 

The results for the two functions $f(z) = \log(1+z)/z$ and $f(z) = \sqrt{z}$ are shown in \cref{fig:compare-bounds--quadform}, where we also include the error for $f(A) \vec b$ for comparison. We can notice that the convergence for the quadratic form $\vec b^T f(A) \vec b$ is roughly twice as fast as the convergence for the matrix-vector product with $f(A)$. 
Similarly to the case of the bounds for $f(A) \vec b$, the a priori bound with exact linear system residual is significantly better than the bound that uses \cref{thm:apriori-linsys-bound--residual-sylvester}. 
Moreover, the a priori bound that uses both the exact residual and the exact error is even more accurate; this is in agreement with the observations in \cref{rem:quadform-bounds-discussion}, where we show that for the best possible choice of $w$, i.e.~$w(z) = z$, the bound in \cref{cor:integral-error-bound--a-priori-quadform} that uses both the error and the residual is more accurate than the bound that only uses $\norm{\res_m(z)}_2$. In the left panel of \cref{fig:compare-bounds--quadform}, we can see that in the last iteration some of the bounds are below the error curve because the error in the Krylov approximation is starting to stagnate at about $10^{-14}$. 

\begin{figure}[h]
	\centering
	\includegraphics[width=0.49\textwidth]{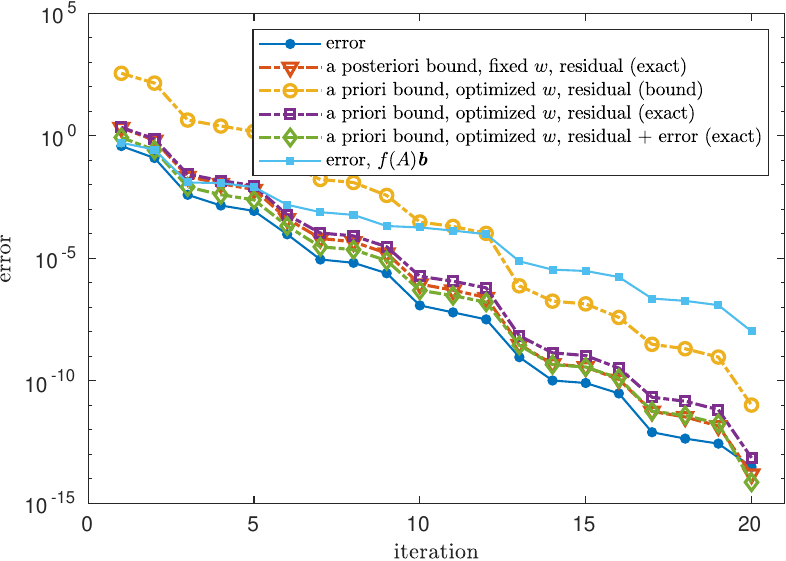}%
	\:%
	\includegraphics[width=0.49\textwidth]{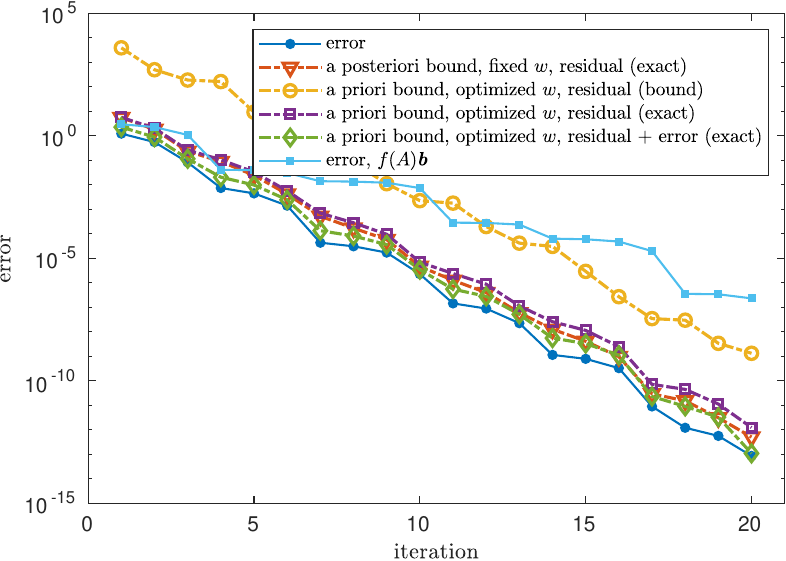}
	\caption{A priori and a posteriori bounds for $\vec b^T f(A) \vec b$, where $A$ is a $1000 \times 1000$ symmetric matrix with logspaced eigenvalues in $[10^{-2}, 10^2]$. Left: $f(z) = \log(1+z)/z$. Right: $f(z) = \sqrt{z}$.}
	\label{fig:compare-bounds--quadform}
\end{figure}

\subsection{Comparison with Lanczos setting} 
In this section we replicate the experiment in~\cite[Figure~4b]{CGMM21} to show that the optimization of $w$ can also lead to improvements for the error bounds in the polynomial Krylov (Lanczos) setting.

We consider the computation of $A^{1/2} \vec b$, where $A$ is a $1000 \times 1000$ symmetric matrix with uniformly spaced eigenvalues in $[10^{-2}, 10^2]$, using a polynomial Krylov method.
The a priori and a posteriori bounds used in \cite{CGMM21} can be obtained as a special case of the error-based bounds in \cref{cor:integral-error-bound--a-priori,cor:integral-error-bound--a-posteriori}, by taking $q_{m-1}(z) \equiv 1$ and fixing $w = 0$. 
We compare these bounds with the ones obtained by optimizing $w$ as described in \cref{subsec:a-priori-bounds}. Note that since we are using $\norm{\err_m(w)}_2$ instead of $\norm{\res_m(w)}_2$ in the a posteriori bound, the discussion in \cref{subsec:a-posteriori-bounds} does not apply and the bound can be improved by optimizing the choice of~$w$.   
Similarly to \cite{CGMM21}, we assume that the error norms $\norm{\err_m(w)}_2$ are known exactly. This assumption allows us to theoretically investigate the error bounds by eliminating any dependence on possibly inaccurate bounds on the linear system error norms, although it is not realistic in a practical scenario.

The results of the comparison are shown in \cref{fig:compare-bounds--poly}, where the optimized bounds are plotted using the same color as the corresponding bound from \cite{CGMM21} with fixed $w = 0$. We can see that optimizing $w$ gives significant improvements, especially for the a priori bound; the optimized a posteriori bound is practically overlapping with the error curve. 
It should be mentioned that the plots in \cite[Figure~4]{CGMM21} consider the $A$-norm of the error, while here we consider the $2$-norm, so the error curves shown in \cref{fig:compare-bounds--poly} are slightly different from the ones in \cite{CGMM21}.   

\begin{figure}[h]
	\centering
	\includegraphics[width=0.65\textwidth]{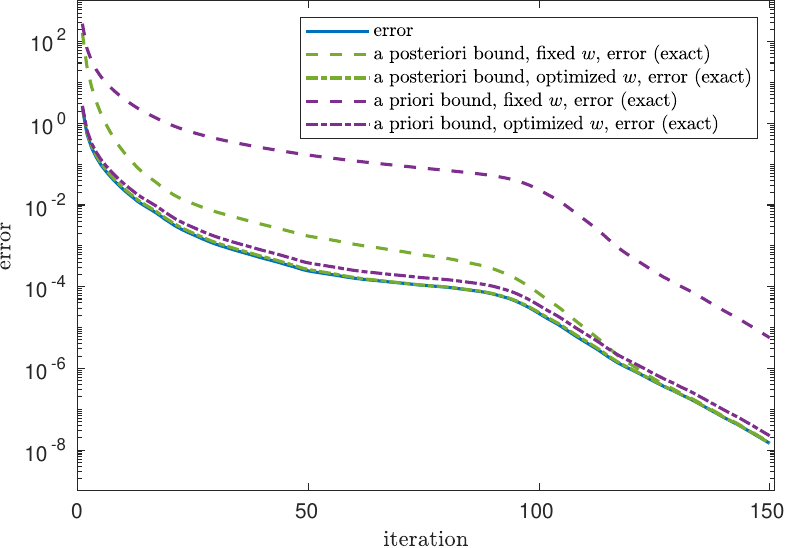}%
	\caption{Comparison of bounds for $A^{1/2} \vec b$, where $A$ is a $1000 \times 1000$ symmetric matrix with linearly spaced eigenvalues in $[10^{-2}, 10^2]$, with and without optimization of $w$.}
	\label{fig:compare-bounds--poly}
\end{figure}

\section{Conclusions}
\label{sec:conclusions}

We have derived error bounds for the approximation of $f(A) \vec b$ and $ \vec b^T f(A) \vec b$ with rational Krylov methods, by exploiting properties of rational Arnoldi decomposition and an integral representation of the function $f$. The bounds that we have obtained generalize some bounds for Lanczos-based matrix function approximation given in \cite{CGMM21}. In the rational Krylov setting, the more complicated expression for the matrix function error poses an additional challenge compared to the polynomial Krylov case; we were able to overcome this obstacle by numerically optimizing the parameter $w$ that appears in the bounds. The same strategy can be also applied to the Lanczos setting, and it can lead to improvements in those bounds as well.

\bibliographystyle{siam}
\bibliography{manuscript-biblio}

\end{document}